\documentclass[11 pt,reqno]{amsart}

\usepackage[utf8]{inputenc}
\usepackage{amsmath,amssymb,amsfonts,mathtools,amsthm}
\usepackage{graphicx}
\usepackage{graphics}
\usepackage{amstext} 
\usepackage[hidelinks,pdfencoding=auto,psdextra]{hyperref}
\usepackage{color,soul}
\usepackage[dvipsnames]{xcolor}

\usepackage{float}
\usepackage{tabularray}
\usepackage{tasks}
\usepackage{tikz-cd}
\usetikzlibrary{graphs,quotes}

\usepackage{geometry}
\setul{}{0.3 ex}

\newtheorem{theorem}{Theorem}

\newtheorem{lemma}[theorem]{Lemma}

\newtheorem{proposition}[theorem]{Proposition}

\theoremstyle{definition}

\newtheorem*{remark}{Remark}
\newtheorem{example}[theorem]{Example}

\title[Totally geodesic Lagrangian submanifolds of $\Sl\times\Sl$]{Totally geodesic Lagrangian submanifolds of the pseudo-nearly Kähler $\Sl\times\Sl$}
\date{}
\author{Mateo Anarella and Joeri Van der Veken}
\address{M. Anarella, Address 1: Université polytechnique Hauts-de-France, Campus Mont Houy
59313 Valenciennes Cedex 9, France. Address 2: KU Leuven, Department of Mathematics, Celestijnenlaan 200 B – Box 2400, 3001 Leuven, Belgium}
\email{mateo.anarella@kuleuven.be}
\address{J. Van der Veken, KU Leuven, Department of Mathematics, Celestijnenlaan 200 B – Box 2400, 3001 Leuven, Belgium}
\email{joeri.vanderveken@kuleuven.be}
\thanks{M.\ Anarella is supported by Methusalem grant METH/21/03 – long term structural funding of the Flemish Government. 
J.\ Van der Veken is supported by the Research Foundation--Flanders (FWO) and the National Natural Science Foundation of China (NSFC) under collaboration project G0F2319N, by the KU Leuven Research Fund under project 3E210539 and by the Research Foundation--Flanders (FWO) and the Fonds de la Recherche Scientifique (FNRS) under EOS Projects G0H4518N and G0I2222N}

\newcommand{\R}{\mathbb{R}}
\newcommand{\Sl}{\mathrm{SL}(2,\R)}
\newcommand{\id}{\operatorname{Id}}
\newcommand{\li}{\langle}
\newcommand{\ri}{\rangle}

\newcommand{\ii}{\textit{\textbf{i}}}
\newcommand{\jj}{\textbf{\textit{j}}}
\newcommand{\kk}{\textbf{\textit{k}}}
\newcommand{\Ss}{\mathbb{S}}
\newcommand{\slf}{\mathfrak{sl}(2,\R)}

\newcommand{\SO}{\text{SO}}
\newcommand{\Slt}{\mathrm{SL}(3,\R)}
\newcommand{\U}{\text U}
\newcommand{\SU}{\text{SU}}
\newcommand{\isoo}{\operatorname{Iso_o}}
\newcommand{\iso}{\operatorname{Iso}}

\makeatletter
\newcommand\columntag[2]{#1\def\@currentlabel{#1}.\label{#2}}
\makeatother

\usepackage{import}
\usepackage{xifthen}
\usepackage{pdfpages}
\usepackage{transparent}

\begin{document}

\begin{abstract}
   In this paper, we study Lagrangian submanifolds of the pseudo-nearly Kähler $\Sl\times\Sl$.
   First, we show that they split into four different classes depending on their behaviour with respect to a certain almost product structure on the ambient space.
   Then, we give a complete classification of totally geodesic Lagrangian submanifolds of this space. 
\end{abstract}
\maketitle
\section{Introduction}

\vspace{3 ex}

The concept of a nearly Kähler manifold was introduced in 1959 by Tachibana in \cite{tachibana}. In the 1970s, Gray extended the study of nearly Kähler manifolds in \cite{gray1,gray2}  and together with Hervella they showed the importance of these spaces in \cite{gray3}. Later on, the work of Nagy \cite{nagy} exhibited the special role of homogeneous six-dimensional nearly Kähler manifolds.
It was not until 2002 that Butruille \cite{butruille} classified, up to isometries, all homogeneous Riemannian six-dimensional strictly nearly Kähler manifolds, these being $\Ss^6$, $\Ss^ 3\times \Ss^ 3$, $\mathbb{C}P^3$ and the space $F(\mathbb{C}^3)$ of full flags in $\mathbb{C}^3$.
Six similar examples of pseudo-Riemannian nearly Kähler manifolds (see diagram below), or pseudo-nearly Kähler manifolds for short, were introduced in \cite{ines,Schafer}, but a complete classification has not been found yet.

\[
\begin{tikzcd}[ampersand replacement=\&, row sep=small, column sep=small]
    \Ss^3\times\Ss^3\ar[dd] \&
    \Ss^6\ar[dd] \&
    \mathbb{C}P^3\arrow[dr, to path=|- (\tikztotarget)]
    \arrow[ddr, to path=|- (\tikztotarget)]
    \&
    \&
    F(\mathbb{C}^3)\arrow[dr, to path=|- (\tikztotarget)]
    \arrow[ddr, to path=|- (\tikztotarget)] \& \\
    \&
    \&
    \&
    \frac{\SO^+(2,3)}{\U(1,1)} \&  
    \&
    \frac{\SU(2,1)}{\U(1)\times\U(1)}  \\
    \Sl\times\Sl \&
    \Ss^6_4 \&
     \&
    \frac{\SO^+(4,1)}{\U(2)}\&
     \&
    \frac{\Slt}{\R^* \times \SO(2)}
\end{tikzcd}
\]
Here, $\Ss^6_4$ means the elements of $\R^7_4$ with length 1. The group $\SO(p,q)$ is not connected in general, thus the symbol $+$ represents the connected component of the identity. 

    
The pseudo-nearly Kähler $\Sl\times\Sl$ is an analogue of $\Ss^ 3\times \Ss^ 3$, as the split quaternions can be used to replace the quaternions to describe its geometric structure. 

When studying submanifolds of an almost Hermitian manifold, two types of submanifolds attract our attention: the almost complex and the totally real submanifolds. 
The former are those submanifolds whose tangent spaces are preserved by the almost complex structure of the ambient space. 
The latter are those whose tangent spaces are mapped into the normal spaces by the almost complex structure of the ambient space.
If a totally real submanifold has half the dimension of the ambient manifold, it is called a Lagrangian submanifold. 

In \cite{wang} the authors showed that any Lagrangian submanifold of a six-dimensional strictly nearly Kähler manifold that has parallel second fundamental form, must have vanishing second fundamental form, i.e. it has to be totally geodesic. The tools the authors used, can also be applied in the pseudo-Riemannian case (see \cite{Schafer}). In the same paper, the authors gave a complete classification of totally geodesic Lagrangian submanifolds of the nearly Kähler $\Ss^ 3\times \Ss^ 3$.

In this work, we study Lagrangian submanifolds of the pseudo-nearly Kähler $\Sl\times\Sl$ and we divide them in four different groups, according to their behavior with respect to an almost product structure on $\Sl\times\Sl$.
If we add the condition of being totally geodesic, we find that these submanifolds can only be part of the first group, also called Lagrangian submanifolds of diagonalizable type.
The full classification of totally geodesic Lagrangian submanifolds of $\Sl\times\Sl$ is as follows.
\begin{theorem}\label{maintheorem}
Any totally geodesic Lagrangian submanifold of the pseudo-nearly Kähler $\Sl\times\Sl$ is congruent to the image of one of the following maps, possibly restricted to an open subset:

\begin{enumerate}
    \item $f\colon\Sl\to\Sl\times\Sl\colon u\mapsto(\id_2,u)$, \label{map1}
        \item $f\colon\Sl\to\Sl\times\Sl\colon u\mapsto(u,\ii u\ii)$, \label{map2}
    \item $f\colon\Sl\to\Sl\times\Sl\colon u\mapsto(u,\kk u\kk)$, \label{map3}
\end{enumerate}
where $\id_2,\ii,\kk$ are the matrices
\[
\id_2=\begin{pmatrix}
1&0\\
0&1\\
\end{pmatrix}, \ \ \ \ \ \ \ii=\begin{pmatrix}
0&1\\
-1&0\\
\end{pmatrix}, \ \ \ \ \ \
\kk=\begin{pmatrix}
1&0\\
0&-1
\end{pmatrix}.
\]

Conversely, the maps \emph{(\ref{map1})}, \emph{(\ref{map2})} and \emph{(\ref{map3})} are totally geodesic Lagrangian immersions.
\end{theorem}

In \cite{wang}, the authors stated that any totally geodesic Lagrangian submanifold of $\Ss^3\times\Ss^3$ is congruent to an immersion of a list of 6 examples. 
This classification can be simplified, using isometries equivalent to the ones given in \eqref{isoslsl}. 
This way, the list is reduced to just two examples, similar to immersions \eqref{map1} and \eqref{map2} of Theorem \ref{maintheorem}. 
Hence, immersion \eqref{map3} is a new example, which arises from the pseudo-Riemannian nature of $\Sl\times\Sl$.

We can understand the geometry of the immersions of Theorem \ref{maintheorem} via the identification of $\Sl$ with the anti-de Sitter space $H_1^3(-1)$; namely via the map  
\begin{equation*}
    H_1^3(-1)\subset \R_2^4\to\Sl:(x_0,x_1,x_2,x_3)\mapsto \begin{pmatrix}
    x_0-x_2 &x_3-x_1\\
    x_3+x_1 & x_0+x_2
\end{pmatrix},
\end{equation*} which is an isometry between $H^3_1(-1)$ and $\Sl$ with the metric introduced in Equation \eqref{prodsl2} below. Here, $\R^4_2$ denotes $\R^4$ equipped with the indefinite inner product 
\[\li x,y\ri=-x_0y_0-x_1y_1+x_2y_2+x_3y_3,\] 
and the three-dimensional anti-de Sitter space with constant sectional curvature $c<0$ is defined as $H^3_1(c)=\{x\in\R^4:\li x,x\ri=1/c\}$.

Note that the three immersions of Theorem \ref{maintheorem} induce essentially different Riemannian structures on $\Sl$. The first immersion induces a metric with constant sectional curvature, that is homothetic to the standard metric. The second immersion turns $\Sl$ into an anti-de Sitter space with a Berger-like metric stretched in a timelike direction, and the third immersion turns it again into an anti-de Sitter space with a Berger-like metric, but now stretched in a spacelike direction. These metrics have been studied more generally in \cite{calvaruso_helix_2023} and \cite{calvaruso_metrics_2014}.

The paper is organised as follows. In Section \ref{preliminaries} we recall the nearly Kähler structure of $\Sl\times\Sl$ and we give a comparison with the product metric. In Section \ref{lagrangiansubmanifolds} we state some properties of Lagrangian submanifolds and we divide them into four groups.
Finally, in Section \ref{totallygeodesic} we classify, up to congruence, all totally geodesic Lagrangian submanifolds of the nearly Kähler $\Sl\times\Sl$.

\section{The pseudo-nearly Kähler \texorpdfstring{$\Sl\times\Sl$}{SL(2,R)xSl(2,R)}} \label{preliminaries}
We recall the geometric structure of the pseudo-nearly Kähler $\Sl\times\Sl$. Most of this section is based on \cite{Ghandour}.
\subsection{Nearly Kähler manifolds}  A nearly Kähler manifold is an almost Hermitian manifold for which the tensor $G(X,Y)=(\tilde{\nabla}_XJ)Y$ is skew symmetric. Here, $J$ is the almost complex structure and $\tilde{\nabla}$ is the Levi-Civita connection. For a (pseudo-)nearly Kähler manifold, the tensor $G$ satisfies
\begin{equation}
    g(G(X,Y),Z)+g(G(X,Z),Y)=0, \ \ \ \ \ G(X,JY)+JG(X,Y)=0,\label{nkprop}
\end{equation}
where $g$ is the metric and $X$, $Y$ and $Z$ are vector fields on the manifold.
As an immediate consequence we have 
\begin{equation}
    g(G(X,Y),JZ)+g(G(X,Z),JY)=0.
    \label{gnormal}
\end{equation}

\subsection{The manifold \texorpdfstring{$\Sl$}{SL(2,R)}}Now consider the following non-degenerate indefinite inner product on $\R^4$:
\begin{equation*}
    \langle a,b\rangle =-\frac{1}{2}(a_1b_4-a_2b_3-a_3b_2+a_4b_1). 
\end{equation*}
We can identify the $2\times 2$ real matrices space $M(2,\R)$ with $\R^4$, so the above inner product can be seen as 
\begin{equation}
    \langle a,b\rangle=-\frac{1}{2}\operatorname{Trace}(\operatorname{adj}(a)b),\label{prodsl2}
\end{equation}
where $\operatorname{adj}(a)$ is the adjugate matrix of $a$. 
The real special linear group $\Sl$ is the set of all the $2\times 2$ real matrices with determinant 1, which turns out to be 
\[
\Sl=\{a\in M(2,\R):\langle a,a\rangle=-1\},
\]
thus $\Sl$ inherits the Lorentzian metric in \eqref{prodsl2} from $M(2,\R)\cong\R^4_2$, which has constant sectional curvature $-1$. This way, we can think of $\Sl$ as the anti-de Sitter space $H_1^{3}(-1)$. The tangent space of $\Sl$ at a matrix $a$ can be written as
\begin{equation}
T_a\Sl =\{ a\alpha: \alpha\in \mathfrak{sl}(2,\R)\},    \label{tangentspace}
\end{equation}
where $\mathfrak{sl}(2,\R)$ is the Lie algebra of $\Sl$, namely the set of all the matrices in $M(2,\R)$ with vanishing trace.

The three-dimensional vector space $\mathfrak{sl}(2,\R)$ is spanned by the pseudo-orthonormal basis composed by the so called split quaternions $\ii,\jj,\kk$, given by
\begin{equation*}
\ii=\begin{pmatrix}
0&1\\
-1&0\\
\end{pmatrix}, \ \ \ \ \ \ \jj=\begin{pmatrix}
0&1\\
1&0\\
\end{pmatrix}, \ \ \ \ \ \
\kk=\begin{pmatrix}
1&0\\
0&-1
\end{pmatrix}.\label{splitquaterions}
\end{equation*}

\subsection{The nearly Kähler structure on \texorpdfstring{$\Sl\times\Sl$}{SL(2,R)xSL(2,R)}}
We define an almost complex structure $J$ on $\Sl\times\Sl$ by

\[
J(a\alpha,b\beta)=\frac{1}{\sqrt{3}} (a(\alpha-2\beta),b(2\alpha-\beta)),
\]
for $\alpha,\beta\in\mathfrak{sl}(2,\R)$. It is well known that with any given (indefinite) metric and almost complex structure we can construct a metric in such a way that the almost complex structure is compatible with the new metric. Starting from the product metric, we get an explicit expression for the new metric:
\begin{equation*}
    g((a\alpha,b\beta),(a\gamma,b\delta))=\frac{2}{3}\langle(a\alpha,b\beta),(a\gamma,b\delta)\rangle-\frac{1}{3}\langle(a\beta,b\alpha),(a\gamma,b\delta)\rangle,
\end{equation*}
for $\alpha,\beta,\gamma,\delta\in\mathfrak{sl}(2,\R) $, and $\li,\ri$ is the product metric of $\Sl\times\Sl$ associated to the metric of $\Sl$ given in (\ref{prodsl2}).

As above we denote the covariant derivative of $J$ by $G$ and we can see that $G$ is skew symmetric, so $\Sl\times\Sl$ together with $J$ and $g$ forms a pseudo-nearly Kähler manifold. As a consequence $G$ satisfies (\ref{nkprop}).

\subsection{The almost product structure \texorpdfstring{$P$}{P}}
Consider the almost product structure $P$ on $\Sl\times\Sl$ defined by
\begin{equation}
    P(a\alpha,b\beta)=(a\beta,b\alpha). \label{prodstructuredef}
\end{equation}

The tensor $P$ has the following properties:

\begin{equation}
    \begin{split}
        &P^2=\id, \ \ \ \ \ \ \ \ \ \ \ \ \ \ \ \ \ \ \ \ \ g(PX,PY)=g(X,Y),\\
        &PJ=-JP, \ \ \ \ \ \ \ \ \ \ \ \ \ \ \ \ \ g(PX,Y)=g(X,PY),
    \end{split}\label{proppe}
\end{equation}
for any vector fields $X,Y$ on $\Sl\times\Sl$.

The tensor $G$ has an explicit expression which is given in the following proposition.
\begin{proposition}[Proposition 3.2 in \cite{Ghandour}]
Let $X=(a\alpha,b\beta),Y=(a\gamma,b\delta)\in T_{(a,b)}(\Sl\times\Sl)$. Then
\begin{equation*}
    G(X,Y)=\frac{2}{3\sqrt{3}}(a(-\alpha\times\gamma-\alpha\times\delta+\gamma\times\beta+2\beta\times\delta),b(-2\alpha\times\gamma+\alpha\times\delta-\gamma\times\beta+\beta\times\delta)),
\end{equation*}
where $\times$ is defined on $\mathfrak{sl}(2,\R)$ by $\alpha\times\beta=\frac{1}{2}(\alpha\beta-\beta\alpha)$.
\label{tensorg}
\end{proposition}
Notice that using this proposition now we can easily see that 
\begin{equation}
    PG(X,Y)+G(PX,PY)=0.
    \label{pyg}
\end{equation}
Also from \cite{Ghandour} we see that $\Sl\times\Sl$ has constant type $-\tfrac{2}{3}$. That is 
\begin{equation}
    \begin{split}
     g(G(X,Y),G(Z,W))&=-\tfrac{2}{3}\big(g(X,Z)g(Y,W)-g(X,W)g(Y,Z)\\
     &\quad+g(JX,Z)g(Y,JW)-g(JX,W)g(Y,JZ)\big).
    \end{split}
    \label{constanttype}
    \end{equation}
We denote by $\tilde{\nabla}$ the Levi-Civita connection on $\Sl\times\Sl$ associated to $g$.  The curvature tensor $\tilde{R}$ of $\Sl\times\Sl$ associated to $\tilde{\nabla}$ is given by
\begin{equation}
        \begin{split}
            \Tilde{R}(U,V)W&=-\tfrac{5}{6}\Big(g(V,W)U-g(U,W)V\Big)\\
            &\quad-\tfrac{1}{6}\Big(g(JV,W)JU-g(JU,W)JV-2g(JU,V)JW\Big)\\
            &\quad-\tfrac{2}{3}\Big(g(PV,W)PU-g(PU,W)PV\\
            &\quad+g(JPV,W)JPU-g(JPU,W)JPV\Big).\label{curv}
        \end{split}
    \end{equation}

\subsection{The isometries of \texorpdfstring{$\Sl \times\Sl$}{SL(2,R)xSL(2,R)}} The contents of this subsection cannot be found in the literature. However, Moruz and Vrancken showed in \cite{properties} a similar result for $\Ss^3\times\Ss^3$, which can be reproduced for $\Sl\times\Sl$.

The nearly Kähler metric $g$ may also be defined from a homogeneous point of view. 
Consider the triple product $\Sl\times\Sl\times\Sl$ with the product metric that arises from the one in \eqref{prodsl2}, and the submersion 
$\pi:\Sl\times\Sl\times\Sl\to\Sl\times\Sl$ given by $\pi(a,b,c)=(ac^{-1},bc^{-1})$. 
This way, $g$ is the metric on $\Sl\times\Sl$ that makes $\pi$ into pseudo-Riemannian submersion.

The connected component of the isometry group of $\Sl\times\Sl$ is 
\[\isoo(\Sl\times\Sl)=\Sl\times\Sl\times\Sl,\] 
where an element $\phi_{(a,b,c)}$ acts on a point $(p,q)$ by $\phi_{(a,b,c)}(p,q)=(apc^{-1},bqc^{-1})$.
 Consequently, $\Sl\times\Sl$ is a homogeneous pseudo-Riemannian manifold:
\[
    \Sl\times\Sl=\frac{\Sl\times\Sl\times\Sl}{\Delta\Sl},
\]
where $\Delta \Sl=\{(a,a,a):a\in\Sl\}$.

From permutations of the factors of $\Sl\times\Sl\times\Sl$ we obtain six different isometries of $\Sl\times\Sl$:
\begin{equation}
    \label{isoslsl}
    \begin{alignedat}{2}
        &\Psi_{0,0}(p,q)=(p,q), 
        &&\Psi_{1,0}(p,q)=(q,p),\\
        &\Psi_{0,2\pi/3}(p,q)=(p q^{-1},q^{-1}),
        &&\Psi_{1,2\pi/3}(p,q)=(q^{-1},p q^{-1}),\\
        &\Psi_{0,4\pi/3}(p,q)=(q p^{-1},p^{-1}),\qquad\qquad
        &&\Psi_{1,4\pi/3}(p,q)=(p^{-1},q p^{-1}).
    \end{alignedat}
\end{equation}
Each one of these isometries is in a different connected component of $\iso(\Sl\times\Sl)$ and satisfies
\[
    J \circ d\Psi_{\kappa,\tau}=(-1)^\kappa  d\Psi_{\kappa,\tau}\circ J,
\ \ \ \ \ \ \ P\circ d \Psi_{\kappa,\tau}=d\Psi_{\kappa,\tau}\circ(\cos\tau P+\sin \tau J P).
\]


    
    \subsection{Comparison with the product metric}
Let $\li,\ri$ be the product metric associated to the metric given in (\ref{prodsl2}), with Levi-Civita connection $\nabla^E$. Here $E$ stands for Euclidean, as the product metric is inherited from $\R^8_4$. We can write $\li,\ri$ in terms of the nearly Kähler metric $g$ and the almost product structure $P$ as
\begin{equation}
    \li X,Y\ri=2 g(X,Y)+g(X,PY), \label{prodmetric}
\end{equation}
    and the connection $\nabla^E$ in terms of the pseudo-nearly Kähler connection $\tilde\nabla$, $J,P$ and $G$ as
\begin{equation}
    \nabla^E_XY=\tilde{\nabla}_XY+\frac{1}{2}(JG(X,PY)+JG(Y,PX)).\label{relprodkal}
\end{equation}
A natural almost product structure for a product manifold that is compatible with the product metric is $Q$, given by
\[
Q(X_1,X_2)=(-X_1,X_2).
\]
Both product structures are related by
\begin{equation}
    QX=-\frac{1}{\sqrt{3}}(2PJX-JX).\label{prodQ}
\end{equation}
Let $D$ be the Euclidean connection for $\R^8_4$, then for $(a,b)\in\Sl\times\Sl$ we have:
\begin{equation*}
\begin{split}
        D_XY&=\nabla_X^EY+\frac{\langle D_XY,(a,b)\rangle}{\langle (a,b),(a,b)\rangle}(a,b)+\frac{\langle D_XY,(-a,b)\rangle}{\langle (-a,b),(-a,b)\rangle}(-a,b)\\
        &=\nabla_X^EY-\frac{1}{2}\langle D_XY,(a,b)\rangle(a,b)-\frac{1}{2}\langle D_XY,(-a,b)\rangle(-a,b).\\
\end{split}
\end{equation*}
As $\langle Y,(a,b)\ri=0$, this implies
\begin{equation}
\begin{split}
        D_XY&=\nabla_X^EY+\frac{1}{2}\langle X,Y\rangle(a,b)+\frac{1}{2}\langle Y,QX\rangle(-a,b). \label{connectionr8}
\end{split}
\end{equation}

\section{Lagrangian submanifolds of \texorpdfstring{$\Sl\times\Sl$}{SL(2,R)xSL(2,R)}}\label{lagrangiansubmanifolds}

We start by recalling some general facts from submanifold theory.
Let $f:M\to N$ be a non-degenerate pseudo-Riemannian immersion. The Gauss formula gives a relation between the Levi-Civita connection of the ambient space and the Levi-Civita connection of the submanifold as follows:
\begin{equation}
    \tilde{\nabla}_XY=\nabla_XY+h(X,Y),\label{gaussformula}
\end{equation}
where $X,Y$ are vector fields on $M$ and $h$ is a symmetric bilinear normal form called the second fundamental form. 
If $h$ vanishes everywhere then $M$ is said to be totally geodesic.

 A relation between $\tilde{\nabla}$ and the normal connection, is provided by the Weingarten formula:
 \[
 \tilde{\nabla}_X\xi=-S_{\xi}X+\nabla_X^{\bot}\xi
 \]
where $X$ and $\xi$ are tangent and normal vector fields on $M$, respectively. For a normal vector field $\xi$, the tensor $S_\xi$ is a symmetric (with respect to $g$) tensor called the shape operator, which satisfies $g(S_\xi X,Y)=g(h(X,Y),\xi)$ and it is linear at $\xi$ and $X$.

The Gauss and Codazzi equations are given by
\begin{equation*}
    \begin{split}
        \left(\tilde{R}(X,Y)Z\right)^{\top}&=R(X,Y)Z+S_{h(X,Z)}Y-S_{h(Y,Z)}X,\\
        \left(\tilde{R}(X,Y)Z\right)^{\bot}&=(\overline{\nabla}_Xh)(Y,Z)-(\overline{\nabla}_Yh)(X,Z),
    \end{split}
\end{equation*}
where $(\overline{\nabla}_Xh)(Y,Z)=\nabla_X^{\bot}h(Y,Z)-h(\nabla_XY,Z)-h(Y,\nabla_XZ)$ and $\tilde{R},R$ are the curvature tensors of $N$ and $M$, respectively.

Let us now consider Lagrangian submanifolds. 
A submanifold $M^n$ of an almost Hermitian manifold $(N^{2n},g,J)$ is said to be Lagrangian if $J$ maps tangent spaces of $M$ into normal spaces, and vice versa.
We will also assume that the submanifold is non-degenerate, i.e. there does not exist any vector field $X$ on $M$ such that $g(X,Y)=0$ for all vector fields $Y$ on $M$.

Given a Lagrangian immersion into a (pseudo-)nearly Kähler manifold,  we extract from \cite{Schafer} some properties of the tensor $G$ and the second fundamental form:
\begin{equation}
    \begin{split}
    g(G(X,Y),Z)&=0,\\
    g(h(X,Y),JZ)&=g(h(X,Z),JY).\\
    \end{split}
    \label{lagrprop}
\end{equation}

The proof of the following result can be found in \cite{Schafer} as well.
\begin{proposition}
Any Lagrangian submanifold of a six-dimensional strictly pseudo-nearly Kähler  \hyphenation{ma-ni-fold} manifold is orientable and minimal. \label{minimal}
\end{proposition}

Consider again the tensor $P$. As the submanifold is Lagrangian, we can write the tangent bundle of the ambient space as $T(\Sl\times\Sl)=TM\oplus JTM$, so there exist two endomorphisms $A,B:TM\rightarrow TM$ such that $P|_M=A+JB$. From Equation (\ref{proppe}) we deduce that these operators are symmetric with respect to $g$, commute with each other and $A^2+B^2=\id$. 
Then, the Gauss and Codazzi equations follow from (\ref{curv}) as:
\begin{equation}
    \begin{split}
        R( X,Y)Z&=-\tfrac{5}{6}\big(g(Y,Z) X-g( X,Z)Y\big)\\
            &\quad -\tfrac{2}{3}\big(g(AY,Z)A X-g(A X,Z)AY+g(BY,
            Z)B X\\
            &\quad-g(B X,Z)BY\big)-S_{h( X,Z)}Y+S_{h(Y,Z)} X,
            \label{Gauss}
    \end{split}
\end{equation}
\begin{equation}
\begin{split}
    (\overline{\nabla}_ X h)(Y,Z)-(\overline{\nabla}_Y h)( X,Z)&=-\tfrac{2}{3}\big(g(AY,Z)JB X-g(A X,Z)JBY\\
    &\quad-g(BY,Z)JA X+g(B X,Z)JAY\big).\label{Codazzi}
\end{split}
\end{equation}

A basis $\{e_1,e_2,e_3\}$ of $\R^3_1$ is said to be $\Delta_i$-orthonormal if the matrix of inner products is given by $\Delta_i$, where
\begin{equation*}
\Delta_1=\begin{pmatrix}
    -1 & 0 & 0 \\
    0 & 1 & 0 \\
    0  & 0 & 1 \\
\end{pmatrix},\ \ \ \
\Delta_2=\begin{pmatrix}
    0 & 1 & 0 \\
    1 & 0 & 0 \\
    0  & 0 & 1 \\
\end{pmatrix},\ \ \ \ 
\Delta_3=\begin{pmatrix}
    1 & 0 & 0 \\
    0 & -1 & 0 \\
    0  & 0 & 1 \\
\end{pmatrix} .
\end{equation*}

A positive oriented frame on a Lorentzian manifold $M$ is said to be a $\Delta_i$-orthonormal frame if it is a $\Delta_i$-orthonormal basis at each point.

Notice that symmetric operators no longer necessarily diagonalize in pseudo-Riemannian ambient spaces, so we have to resort to a result from \cite{Magid} which we will adapt to our particular case.
\begin{lemma}
Let $A$ and $B$ be two symmetric operators with respect to a Lorentzian metric on a three-dimensional vector space $V$. Assume that they commute and that $A^2 + B^2 = \id$. Then $A$ and $B$ must take one of the following forms, with respect to a $\Delta_i$-orthonormal basis.
\begin{table}[H] 
  \makebox[\textwidth]{
  \begin{tabular}{l l l l}
    \columntag{(1)}{case:8.1} \label{cc1} &
    $A = \begin{pmatrix}
    \cos 2\theta_1 & 0 & 0 \\
    0 & \cos 2\theta_2 & 0 \\
    0 & 0 & \cos 2\theta_3
    \end{pmatrix}$, & \hspace*{0.5 cm}&
    $B = \begin{pmatrix}
     \sin 2\theta_1 & 0 & 0 \\
    0 & \sin 2\theta_2 & 0 \\
    0 & 0 & \sin 2\theta_3
    \end{pmatrix}$, \\[4ex]
     \multicolumn{4}{l}{with $\Delta_i=\Delta_1$
    and $\theta_1,\theta_2,\theta_3\in[0,\pi)$. \hspace*{9 cm}}
\end{tabular}
}
\end{table}

    \begin{table}[H] 
        \centering
        \makebox[\textwidth]{
        \begin{tabular}{l l l}
     \columntag{(2)}{case:8.1b} \label{cc1b}&
    $A = \begin{pmatrix}
    \cosh \lambda & 0 & 0 \\
    0 & \cosh \lambda & 0 \\
    0 & 0 & \cos 2\theta
    \end{pmatrix}$, &
    $B = \begin{pmatrix}
   0& \sinh \lambda  & 0 \\
    -\sinh \lambda & 0&0 \\
    0 & 0 & \sin 2\theta
    \end{pmatrix}$, \\[4ex]
     \multicolumn{3}{l}{with $\Delta_i=\Delta_1$,  $\lambda\in\R$
    and $\theta\in[0,\pi)$.}
    \\[2ex]
      \columntag{(3)}{case:8.1c} \label{cc1c}&
    $A = \begin{pmatrix}
    \varepsilon_1 & 0 & 0 \\
    0 & \varepsilon_1 & 0 \\
    0 & 0 & \varepsilon_1
    \end{pmatrix}$, &
    $B = \begin{pmatrix}
   0& \varepsilon_2  & 0 \\
   0 & 0&0 \\
    0 & 0 & 0
    \end{pmatrix}$, \\[4ex]
     \multicolumn{3}{l}{with $\Delta_i=\Delta_2$ and $\varepsilon_1,\varepsilon_2\in \{-1,1\}$.}
    \\[2ex]
    \columntag{(4)}{case:8.2} \label{cc2} &
    $A = \begin{pmatrix}
    \cos 2\theta_1 & \varepsilon & 0 \\
    0 & \cos 2\theta_1 & 0 \\
    0 & 0 & \cos 2\theta_2
    \end{pmatrix}$, &
    $B = \begin{pmatrix}
    \sin 2\theta_1 & \frac{-(c^2 + 2 \varepsilon\cos 2\theta_1)}{2 \sin 2\theta_1} & c \\
    0 & \sin 2\theta_1 & 0 \\
    0 & c & \sin 2\theta_2
    \end{pmatrix}$, 
    \\[4ex]
    \multicolumn{3}{l}{  with $\Delta_i=\Delta_2$, 
    $\theta_1,\theta_2 \in [0,\pi)$, $\theta_1 \ne 0,\pi/2$, $\varepsilon=\pm1$ and $c \in \mathbb R$.
    If $c \ne 0$, then $\cos 2\theta_1 = \cos 2\theta_2$}\\[1ex] \multicolumn{3}{l}{and $\sin 2\theta_1 = -\sin 2\theta_2$.}
    \\[2ex]
    \columntag{(5)}{case:8.3} \label{cc3} &
    $A = \begin{pmatrix}
    -\varepsilon & \varepsilon & 0 \\ 0 & -\varepsilon & 0 \\ 0 & 0 & -\varepsilon
    \end{pmatrix}$, &
    $B = \begin{pmatrix}
    0 & t & \sqrt{2} \\
    0 & 0 & 0 \\
    0 & \sqrt{2} & 0
    \end{pmatrix}$,
    \\[4ex]
    \multicolumn{3}{l}{ with $\Delta_i=\Delta_2$, $\varepsilon=\pm1$ and $t \in \R$.}
    \\[2ex]
    \columntag{(6)}{case:8.4}  \label{cc4}&
    $A = \begin{pmatrix}
    \cos 2\theta & 0 & 1\\
    0 & \cos 2\theta & 0 \\
    0 & 1 & \cos 2\theta
    \end{pmatrix}$, &
    $B = \begin{pmatrix}
    \sin 2\theta & -(\csc 2\theta)^3/2 & -\cot 2\theta \\
    0 & \sin 2\theta & 0 \\
    0 & -\cot 2\theta & \sin 2\theta
    \end{pmatrix}$,
    \\[4ex]
    \multicolumn{3}{l}{with $\Delta_i=\Delta_2$, and $\theta \ne 0,\pi/2$.}
    \\[2ex]
    \columntag{(7)}{case:8.5} \label{cc5} &
    $A = \begin{pmatrix}
    s \cos 2\theta_1 & x & 0 \\
    -x & s \cos 2\theta_1 & 0 \\
    0 & 0 & \cos 2\theta_2
    \end{pmatrix}$, &
    $B = \begin{pmatrix}
    s \sin 2\theta_1 & y & 0 \\
    -y & s \sin 2\theta_1 & 0 \\
    0 & 0 & \sin 2\theta_2
    \end{pmatrix}$,
    \\[4ex]
    \multicolumn{3}{l}{ with $\Delta_i=\Delta_3$, $s=\sqrt{1+x^2+y^2}$, $x \ne 0$, $y\sin 2\theta_1 = -x \cos 2\theta_1$ and $\theta_1,\theta_2 \in [0,\pi)$.}
  \end{tabular}
  }
\end{table}

\label{propABcruda}
\end{lemma}

\begin{proof}
In \cite{Magid} it is shown that in a real $3$-dimensional vector space $V$ equipped with a Lorentzian metric, two symmetric linear transformations $A$ and $B$ that commute, can be put into one of the following forms with respect to $\Delta_i$-orthonormal bases:

\begin{align}
    A &=\begin{pmatrix}
            \lambda_1 & 0 & 0 \\
            0 & \lambda_2 & 0 \\
            0 & 0 & \lambda_3 \\
        \end{pmatrix},  && 
    B =\begin{pmatrix}
        \mu_1 & 0 & 0 \\
        0 & \mu_2 & 0 \\
        0 & 0 & \mu_3 \\
        \end{pmatrix},  && 
    \text{with $\Delta_{i}=\Delta_1$}, &&\label{case1} \\
        A &=\begin{pmatrix}
            \lambda_1 & 0 & 0 \\
            0 & \lambda_1 & 0 \\
            0 & 0 & \lambda_2 \\
        \end{pmatrix},  && 
    B =\begin{pmatrix}
        \mu_1 & \mu_2 & 0 \\
        -\mu_2 & \mu_1& 0 \\
        0 & 0 & \mu_3 \\
        \end{pmatrix},  && 
    \text{with $\Delta_{i}=\Delta_1$}, && \mu_2\neq0,\label{case1to4} \\
   A &=\begin{pmatrix}
            \lambda & 0 & 0 \\
            0 & \lambda & 0 \\
            0 & 0 & \lambda \\
        \end{pmatrix}, && 
    B =\begin{pmatrix}
        \mu_1 & \varepsilon & 0 \\
        0 & \mu_1 & 0 \\
        0 & 0 & \mu_2 \\
        \end{pmatrix},  && 
    \text{with $\Delta_{i}=\Delta_2$}, &&\label{case1to2} 
    \end{align}
    \begin{align}
   A &=\begin{pmatrix}
            \lambda & 0 & 0 \\
            0 & \lambda & 0 \\
            0 & 0 & \lambda \\
        \end{pmatrix}, && 
    B =\begin{pmatrix}
        \mu & 0 & 1 \\
        0 & \mu & 0 \\
        0 & 1 & \mu \\
        \end{pmatrix},  && 
    \text{with $\Delta_{i}=\Delta_2$}, &&\label{case1.3b} \\
    A &=\begin{pmatrix}
            \lambda_1 & \varepsilon & 0 \\
            0 & \lambda_1 & 0 \\
            0 & 0 & \lambda_2 \\
        \end{pmatrix}, && 
    B =\begin{pmatrix}
        \mu_1 & b & c \\
        0 & \mu_1 & 0 \\
        0 & c & \mu_2 \\
        \end{pmatrix}, &&
    \text{with $\Delta_{i}=\Delta_2$}, && c\lambda_1=c\lambda_2, \label{case2,3} \\
    A &=\begin{pmatrix}
            \lambda & 0 & 1 \\
            0 & \lambda & 0 \\
            0 & 1 & \lambda \\
        \end{pmatrix}, &&
    B =\begin{pmatrix}
        \mu & b & c \\
        0 & \mu & 0 \\
        0 & c & \mu \\
        \end{pmatrix}, && \text{with $\Delta_{i}=\Delta_2$}, &&\label{case4} \\
               A&=\begin{pmatrix}
            \alpha & \beta & 0 \\
            -\beta & \alpha & 0 \\
            0 & 0 & \lambda \\
        \end{pmatrix}, &&
        B=\begin{pmatrix}
        \gamma & \delta & 0 \\
        -\delta & \gamma & 0 \\
        0 & 0 & \mu \\
        \end{pmatrix}, && \text{with $\Delta_{i}=\Delta_3$},  && \beta\neq 0,  \label{case5} 
\end{align}
with $\varepsilon=\pm1$.
We will analyse the equation $A^2+B^2=\id$ on each type of matrix separately.

\textit{Type 1 \emph{(\ref{case1})}}:
 Computing $A^2+B^2=\id$ in (\ref{case1}) immediately yields Case \ref{case:8.1} of Lemma \ref{propABcruda}.

\textit{Type 2 \emph{(\ref{case1to4})}}:
It follows from $A^2+B^2=\id$ that $\mu_1=0$, $\lambda_1^2-\mu_2^2=1$ and $\lambda^2_2+\mu_3^2=1$. Hence, we get Case \ref{case:8.1b} of Lemma \ref{propABcruda}.

\textit{Type 3 \emph{(\ref{case1to2})}}:
Case \ref{case:8.1c} is immediate from computing $A^2+B^2=\id$,  as we obtain that $\mu_1=0$. It follows that $\lambda=\pm 1$ and that $\mu_2=0$. 
This is Case \ref{case:8.1c} of Lemma \ref{propABcruda}.

\textit{Type 4 \emph{(\ref{case1.3b})}}:
We easily see that under no conditions $A^2+B^2$ can be equal to the identity in this case.

\textit{Type 5 \emph{(\ref{case2,3})}}:
Computing $A^2+B^2=\id$ in (\ref{case2,3}) yields the equations
\begin{align}
      \lambda_1^2+\mu_1^2&=1,\label{1eq}\\ 
        \lambda_2^2+\mu_2^2&=1, \label{2eq}\\
        2\varepsilon\lambda_1+2b\mu_1+c^2&=0, \label{3eq}\\ 
        c(\mu_1+\mu_2)&=0.\label{4eq}
\end{align}

Suppose $c\neq 0$. Then because of Equation (\ref{4eq}) and $c\lambda_1=c\lambda_2$ we have that $\lambda_1=\lambda_2$ and $\mu_1=-\mu_2$. 
If $\mu_1=0$ then (\ref{1eq}) and (\ref{3eq}) imply that $\lambda_1=\lambda_2=-\varepsilon$ and $c=\sqrt{2}$, which is Case \ref{case:8.3} of Lemma \ref{propABcruda}. 
If instead $\mu_1\neq 0$, by Equation (\ref{1eq}) we can write $\lambda_1=\cos2\theta_1$ and $\mu_1=\sin2\theta_1$, then by (\ref{3eq}) we have $b=-(c^2+2\varepsilon\cos2\theta_1)/(2\sin2\theta_2)$, which leave us with Case \ref{case:8.2} of Lemma~\ref{propABcruda}. 

Now suppose that $c=0$. If $\mu_1=0$ then we get a contradiction from equations (\ref{1eq}) and (\ref{3eq}). 
Therefore $\mu_1$ must be different from zero, and from (\ref{1eq}), (\ref{2eq}) and (\ref{3eq}) we get that $\lambda_1=\cos2\theta_1$, $\lambda_2=\cos2\theta_2$, $\mu_1=\sin2\theta_1$, $\mu_2=\sin2\theta_2$ and $b=-\varepsilon\cot2\theta_1$, which is again Case \ref{case:8.2} of Lemma \ref{propABcruda}.

\textit{Type 6 \emph{(\ref{case4})}}: Computing $A^2+B^2=\id$ yields
\begin{align}
    \lambda^2+\mu^2=1,\label{eqqq1}\\
    \lambda+c\mu=0, \label{eqqq2} \\
     2 b \mu +c^2+1=0. \label{eqqq3} 
\end{align}
From (\ref{eqqq1}) and (\ref{eqqq2}) we easily see that $\mu\neq 0$, thus $c=-\lambda/\mu$. By replacing this in (\ref{eqqq3}) we get that $b=-1/(2\mu^3)$. Finally by (\ref{eqqq1}) we have that $\lambda=\cos2\theta$, $\mu=\sin2\theta$, $c=-\cot2\theta$ and $b=-\csc(2\theta)^3/2$, which is Case \ref{case:8.4} in Lemma \ref{propABcruda}.

\textit{Type 7 \emph{(\ref{case5})}}: Again, we compute $A^2+B^2=\id$ in (\ref{case5}), from which follows
\begin{align}
    \alpha^2-\beta^2+\gamma^2-\delta^2=1, \label{c51} \\
    \alpha\beta+\gamma\delta=0, \label{c52}\\
    \lambda^2+\gamma^2=1. \label{c53}
\end{align}
We can transform (\ref{c51}) into
\begin{equation*}
    \begin{split}
     \left(\frac{\alpha}{\sqrt{1+\beta^2+\delta^2}}\right)^2+\left(\frac{\gamma}{\sqrt{1+\beta^2+\delta^2}}\right)^2=1,\\
    \end{split}
\end{equation*}
thus we obtain that $\alpha=\sqrt{1+\beta^2+\delta^2}\cos\theta$ and $ \gamma=\sqrt{1+\beta^2+\delta^2}\sin\theta$ for some $\theta$. 
Equation~(\ref{c52}) becomes $\sqrt{1+\beta^2+\delta^2}(\beta\cos\theta+\delta\sin\theta)=0$, hence we can conclude that $\delta=-\beta\cot\theta$ since if $\sin\theta=0$ then $\beta=0$, which contradicts the last condition of (\ref{case5}).
 Renaming $x=\beta$, $y=\delta$ and $s=\sqrt{1+x^2+y^2}$ we obtain Case \ref{case:8.5} of Lemma \ref{propABcruda}.
\end{proof}

\begin{remark}
    In \cite{Magid}, only the case $\varepsilon=1$ in equations \eqref{case1to2} and \eqref{case2,3}. However, the case $\varepsilon=-1$ is essentially different, as there is no change of basis preserving the metric that can take one case into the other.
\end{remark}

Now, fixing a point $p\in M$ we know that there exists a basis of $T_pM$ as in Lemma \ref{propABcruda}. Then we can extend locally this basis to a $\Delta_i$-orthonormal frame $\{E_1,E_2,E_3\}$ on $M$.

From the first equation of (\ref{lagrprop}) it follows that $G(X,Y)$ is a normal vector to $M$ for any tangent vectors $X,Y$ on $M$. Let $\{E_1,E_2,E_3\}$ be a $\Delta_i$-orthonormal frame of the tangent space of $M$.
By (\ref{gnormal}) and (\ref{constanttype}) we get $JG(E_j,E_k)=\varepsilon \sqrt{\tfrac{2}{3}}E_l$ with $\varepsilon=\pm1$ depending on $\Delta_i$, as the following table shows.  

\begin{equation}
    \begin{tblr}{|c|c|c|c|}
\hline
      &   \Delta_1 &   \Delta_2  & \Delta_3  \\   
     \hline
     JG(E_1,E_2)  &  \sqrt{\frac{2}{3}}E_3  &  \sqrt{\frac{2}{3}}E_3 & \sqrt{\frac{2}{3}}E_3 \\
     JG(E_1,E_3)  &  -\sqrt{\frac{2}{3}}E_2  &  -\sqrt{\frac{2}{3}}E_1  & \sqrt{\frac{2}{3}}E_2 \\
     JG(E_2,E_3)  &  -\sqrt{\frac{2}{3}}E_1  &  \sqrt{\frac{2}{3}}E_2   & \sqrt{\frac{2}{3}}E_1 \\
     \hline
\end{tblr}\label{tabla}
\end{equation}

With this information we can state a stronger version of Lemma $\ref{propABcruda}$.
 \begin{lemma}
 Let $M$ be a Lagrangian submanifold of $\Sl\times\Sl$ and $P$ the almost product structure given in $(\ref{prodstructuredef})$. 
 Then there exists a Lagrangian submanifold $N$ congruent to $M$ such that the restriction of $P$ to $N$ can be written as $P|_N=A+JB$, where $A,B:TN\to TN$ must have one of the following forms, with respect to a $\Delta_i$-orthonormal frame $\{E_1,E_2,E_3\}$.
\begin{table}[H]
  \centering
  \makebox[\textwidth]{
  \begin{tabular}{l l l l}
    \columntag{(1)}{case:10.1} &
    $A = \begin{pmatrix}
    \cos 2\theta_1 & 0 & 0 \\
    0 & \cos 2\theta_2 & 0 \\
    0 & 0 & \cos 2\theta_3
    \end{pmatrix}$, & \hspace*{2 cm}&
    $B = \begin{pmatrix}
    \sin 2\theta_1 & 0 & 0 \\
    0 & \sin 2\theta_2 & 0 \\
    0 & 0 & \sin 2\theta_3
    \end{pmatrix}$, 
    \\[4ex]
    \multicolumn{4}{l}{
    with $\Delta_i=\Delta_1$ and $\theta_1 + \theta_2 + \theta_3 = 0$ modulo $\pi$.}
\end{tabular}}
\end{table}    
    \begin{table}[H]
        \centering
        \makebox[\textwidth]{
        \begin{tabular}{l l l}
     \columntag{(2)}{case:10.2} &
    $A = \begin{pmatrix}
    \cos 2\theta_1 & 1 & 0 \\
    0 & \cos 2\theta_1 & 0 \\
    0 & 0 & \cos 2\theta_2
    \end{pmatrix}$, &
    $B = \begin{pmatrix}
    \sin 2\theta_1 & -\cot 2\theta_1 & 0 \\
    0 & \sin 2\theta_1 & 0 \\
    0 & 0 & \sin 2\theta_2
    \end{pmatrix}$, 
    \\[4ex]
    \multicolumn{3}{l}{
     with $\Delta_i=\Delta_2$, $2\theta_1 + \theta_2 = 0$ modulo $\pi$ and $\theta_1 \ne 0, \pi/2$.}
    \\[2ex]
    \columntag{(3)}{case:10.3} &
    $A = \begin{pmatrix}
    -\frac12 & 0 & 1 \\ 0 & -\frac12 & 0 \\ 0 & 1 & -\frac12
    \end{pmatrix}$, &
    $B = \pm \begin{pmatrix}
    \frac{\sqrt 3}{2} & \frac{-4}{3\sqrt 3} & \frac{1}{\sqrt 3} \\
    0 & \frac{\sqrt 3}{2} & 0 \\
    0 & \frac{1}{\sqrt 3} & \frac{\sqrt 3}{2}
    \end{pmatrix}$,
    \\[4 ex]
    \multicolumn{3}{l}{with $\Delta_i=\Delta_2$.}
    \\[2ex]
    \columntag{(4)}{case:10.4} &
    $A = \begin{pmatrix}
    \cosh\lambda \cos 2\theta_1 & \sinh\lambda\sin\theta_2 & 0 \\
    -\sinh\lambda\sin\theta_2 & \cosh\lambda \cos 2\theta_1 & 0 \\
    0 & 0 & \cos 2\theta_2
    \end{pmatrix}$, & \\[4ex]
    & $B = \begin{pmatrix}
    \cosh\lambda \sin 2\theta_1 & \sinh\lambda\cos\theta_2 & 0 \\
    -\sinh\lambda\cos\theta_2 & \cosh\lambda \sin 2\theta_1 & 0 \\
    0 & 0 & \sin 2\theta_2
    \end{pmatrix}$,&
    \\[4ex]
    \multicolumn{3}{l}{ with $\Delta_i=\Delta_3$, $2\theta_1 + \theta_2 = 0$ modulo $\pi$, $\theta_2 \ne 0,\pi$ and $\lambda \ne 0$.}
  \end{tabular}
  }
\end{table}
\label{propAB}
 \end{lemma}
We will call $\theta_i$, $\theta$ and $\lambda$ angle functions. 
Lagrangian submanifolds such that $P$ takes the form of Case \ref{case:10.1} of Lemma \ref{propAB} are said to be of diagonalizable type.

 \begin{proof}
The next equation follows from applying $P$ to (\ref{tabla}), using (\ref{pyg}) and the fact that $PJ=-JP$:
\begin{equation}
PE_i=-\alpha JPG(E_j,E_k)=\alpha JG(PE_j,PE_k), \label{pei} 
\end{equation}
where $\alpha$ can be either $\sqrt{\tfrac{3}{2}}$ or $-\sqrt{\tfrac{3}{2}}$, depending on which case of (\ref{tabla}) we are in. For Case \ref{case:8.1}, Case \ref{case:8.1b} and Case \ref{case:8.5} of Lemma \ref{propABcruda} the triple $(i,j,k)$ is a permutation of $(1,2,3)$. For cases \ref{case:8.1c} to \ref{case:8.4} of Lemma \ref{propABcruda} this is not true anymore, and $i$ is equal either to $j$ or to $k$.
We consider the seven cases of Lemma \ref{propABcruda} applied to Equation (\ref{pei}), one by one.

\textit{Case \emph{\ref{case:8.1}}.} Using a similar procedure as in \cite{Dioos}, we get $\theta_1+\theta_2+\theta_3=0$ modulo $\pi$.
 
\textit{Case \emph{\ref{case:8.1b}}.} Taking $i=3$, $j=1$ and $k=2$ in \eqref{pei} and looking at the component of $JE_3$ we conclude that $\theta=0$.
Now we apply to $M$ the isometry $\Psi_{4\pi/3,1}$ given in $\eqref{isoslsl}$. The tensor $P$ restricted to $\Psi_{4\pi/3,1}(M)$ can be written as $\tilde{A}+J\tilde{B}$, where 
\[
\tilde{A}=-\frac{1}{2}A+\frac{\sqrt{3}}{2}B, \ \ \ \ \tilde{B}=\frac{\sqrt{3}}{2}A+\frac{1}{2}B.
\]
This is Case \ref{case:10.4} of Lemma \ref{propAB} with $\theta_1=\theta_2=\tfrac{\pi}{3}$.

\textit{Case \emph{\ref{case:8.1c}}.}
With $i=1$, $j=1$ and $k=3$ in \eqref{pei} we obtain $\varepsilon_1=1$.
As there is no possible change of basis that can transform $\varepsilon_2$ into $1$, we apply either $\Psi_{2\pi/3,1}$ or $\Psi_{4\pi/3,1}$ to $M$, depending on its value.
 Proceeding in the same way as in the previous case, we restrict $P$ to the image of the isometry. 
 After a change of basis, we get Case \ref{case:10.2} of Lemma \ref{propAB}, with $\theta_1=\theta_2=\tfrac{\pi}{3}$ if $\varepsilon_2=1$ and $\theta_1=\theta_2=\tfrac{2\pi}{3}$ if $\varepsilon_2=1$.


 \textit{Case \emph{\ref{case:8.2}}}.
Taking $i=1$, $j=1$ and $k=1$ in (\ref{pei}) yields
\[
\cos2\theta_1=\cos2(\theta_1+\theta_2),\ \ \ \ \ \sin2\theta_1=-\sin2(\theta_1+\theta_2).
\]
Thus,
\begin{equation*}
\begin{split}
\cos2(2\theta_1+\theta_2)&=\cos2\theta_1\cos2(\theta_1+\theta_2)-\sin2\theta_1\sin2(\theta_1+\theta_2)\\
&=\cos^22(\theta_1+\theta_2)+\sin^22(\theta_1+\theta_2)\\
&=1.
\end{split}    
\end{equation*}
Then $2\theta_1+\theta_2=0$ mod $\pi$. 
Looking at the component of $E_3$ in (\ref{pei}) when taking $i=2$, $j=2$ and $k=3$ we obtain $c\sin2\theta_1=0$, therefore $c=0$ as $\sin2\theta_1\neq 0$.

If $\varepsilon=1$, we have Case \ref{case:10.2} of Lemma~\ref{propAB}.

If $\varepsilon=-1$, we would like to transform it into $1$. We apply either $\Psi_{2\pi/3,0}$ or $\Psi_{2\pi/3,1}$ to $M$ and restrict $P$ to the image. The component $\tilde{A}_{12}$ is $1+\sqrt{3} \cot 2 \theta_1$ or $1-\sqrt{3} \cot 2 \theta_1$. 
These two values cannot be negative at the same time, thus we can always choose one to be positive. After a change of basis, we get Case \ref{case:10.2} of Lemma \ref{propAB}.

\textit{Case \emph{\ref{case:8.3}}}.
On the left hand side of (\ref{pei}) we have $PE_1=-E_1$ and on the right hand side we have
\begin{equation*}
\begin{split}
    -\sqrt{\frac{3}{2}}JG(PE_1,PE_3)=-\sqrt{\frac{3}{2}}JG(-E_1,-E_3)=E_1,
\end{split}    
\end{equation*}
which is a contradiction. Therefore, there is no Lagrangian submanifold with such a frame.

\textit{Case \emph{\ref{case:8.4}}.}
Taking the component of $E_3$ on both sides of (\ref{pei})
with $i=2$, $j=2$ and  $k=3$ yields $-2\cos2\theta=1$.
Hence $\theta=\tfrac{\pi}{3}$ or $\theta=\tfrac{2}{3}\pi$. That is
\[
    A=\begin{pmatrix}
    -\frac{1}{2} & 0 & 1 \\
    0 & -\frac{1}{2} & 0 \\
    0  & 1 & -\frac{1}{2} \\
\end{pmatrix}, \ \ 
B=\pm\begin{pmatrix}
    \frac{\sqrt{3}}{2} & -\frac{4}{3\sqrt{3}} & \frac{1}{\sqrt{3}} \\
    0 & \frac{\sqrt{3}}{2} & 0 \\
    0  & \frac{1}{\sqrt{3}} & \frac{\sqrt{3}}{2} \\
\end{pmatrix}.
\]

\textit{Case \emph{\ref{case:8.5}}}.
Taking $i=1$, $j=2$ and $k=3$ in (\ref{pei}) gives us
\begin{align*}
    \cos2\theta_1&=\cos2(\theta_1+\theta_2),\\
       \sin2\theta_1&=-\sin2(\theta_1+\theta_2).
\end{align*}
Therefore as in the previous cases we have $\cos2(2\theta_1+\theta_2)=1$. We take $i=2$, $j=1$ and $k=3$ in (\ref{pei}) and we look at the components of $E_1,JE_1$:
\begin{equation}
    \begin{split}
        x&=-x\cos2\theta_2+y\sin2\theta_2,\\ y&=x\sin2\theta_2+y\cos2\theta_2.\\
    \end{split}\label{reflection}
\end{equation}
Note that the equations in (\ref{reflection}) can be written as the equation $Rv=v$ where $v=(x,y)^t$ and
\begin{equation*}
    R=\begin{pmatrix}
    \cos(\pi-2\theta_2) &  \sin(\pi-2\theta_2)\\
    \sin(\pi-2\theta_2) & -\cos(\pi-2\theta_2)\\
    \end{pmatrix}.
    \end{equation*}
Hence, $v$ is an eigenvector of the matrix $R$ associated to the eigenvalue $1$.
Now, $R$ is a reflection in the plane with respect to the straight line  $s\mapsto s(\cos(\frac{\pi}{2}-\theta_2),\sin(\frac{\pi}{2}-\theta_2))^t=s(\sin\theta_2,\cos\theta_2)^t$, therefore $v$ lies in that subspace. 
So $x=\sinh\lambda\sin\theta_2$, $y=\sinh\lambda\cos\theta_2$ for some $\lambda\in\R$ different from zero. 
Finally, replacing $x$ and $y$ in the matrices of Case \ref{cc5} of Lemma \ref{propABcruda} yields what we desired.
\end{proof}

\section{Totally geodesic Lagrangian submanifolds}\label{totallygeodesic}
Let $M$ be a Lagrangian submanifold of $\Sl\times\Sl$ and let $\{E_1,E_2,E_3\}$ be a frame given by Lemma \ref{propAB}.
As the normal space is spanned by $\{JE_1,JE_2,JE_3\}$, we may define functions $\omega_{ij}^k$ and $h_{ij}^k$ by $\nabla_{E_i}E_j=\sum_k\omega_{ij}^kE_k$ and $h(E_i,E_j)=\sum_kh_{ij}^kJE_k$.
From the second equation of (\ref{lagrprop}) and the compatibility of the connection with the metric, we obtain the following symmetries.

First, for frames $\{E_i\}$ associated to $\Delta_1$ and $\Delta_3$ we have
\[
\delta_k\omega_{ij}^k=-\delta_j\omega_{ik}^j, \ \ \ \ \ \ \ h_{ij}^k=h_{ji}^k=\delta_j\delta_kh_{ik}^j,
\]
where 
\[\delta_i=g( E_i,E_i).\]
This implies that $\omega_{ij}^j=0$ for all $i,j=1,2,3$.

If the frame is associated to $\Delta_2$ we get
 \[\omega_{ij}^k=-\omega_{i\widehat{k}}^{\widehat{j}},\ \ \ \ \ \ \ h_{ij}^k=h_{ji}^k=h_{i\widehat{k}}^{\widehat{j}}\] 
 where $\widehat{2}=1$, $\widehat{1}=2$ and $\widehat{3}=3$. As before, we have that $\omega_{i3}^3=0$. 
 Also,  if $j=1$, $k=2$ or $j=2$, $k=1$ then $\omega_{ij}^k=0$.
 
The almost product structure $P$ is not integrable, but we have a neat expression for its covariant derivative. Namely, using an pseudo-orthonormal frame we can check $(\Tilde{\nabla}_XP)JY=J(\Tilde{\nabla}_XP)Y$, from which it follows that 
\begin{equation}
(\Tilde{\nabla}_XP)Y=\frac{1}{2}(JG(X,PY)+JPG(X,Y)). \label{nablap}
\end{equation}

Equation (\ref{nablap}) is useful since it gives conditions on the components $h_{ij}^k$ and $\omega_{ij}^k$. As said equation depends on $P$, we are forced to divide between the four cases of Lemma \ref{propAB}.

\subsection{Lagrangian submanifolds of diagonalizable type}

Equation (\ref{nablap}) for Case \ref{case:10.1} of Lemma \ref{propAB} yields the following lemma.

\begin{lemma}\label{lemmacase1}Let $M$ be a Lagrangian submanifold of $\Sl\times\Sl$. Assume that the operators $A,B$ take the form of the first case in Lemma \emph{\ref{propAB}} with respect to a $\Delta_1$-orthonormal frame $\{E_1,E_2,E_3\}$. Except for $h_{12}^3$, all the components of the second fundamental form are given by the derivatives of the angle functions $\theta_1,\theta_2$ and $\theta_3$:
\begin{equation}
    E_i(\theta_j)=-\delta_i\delta_jh_{jj}^i,
    \label{anglederi}
\end{equation}
where $\delta_i=g(E_i,E_i)$. Also 
\begin{equation}
   h_{ij}^k\cos(\theta_j-\theta_k)=(\tfrac{1}{\sqrt{6}}\delta_k\varepsilon_{ij}^k-\omega_{ij}^k)\sin(\theta_j-\theta_k),\label{sffc}
\end{equation}
for $j\neq k$.
\end{lemma}
\begin{proof}
Taking $X=E_1$ and $Y=E_2$ in (\ref{nablap}) and comparing the components in $E_1,E_2,E_3,JE_1,JE_2$ and $JE_3$ yields the following six equations
\begin{align*}
    (h_{11}^2 \cos(\theta_1 - \theta_2) +  \omega_{11}^2 \sin(\theta_1 - \theta_2)) \sin(\theta_1 + \theta_2)&=0,\\
    (h_{11}^2 \cos(\theta_1 - \theta_2) + \omega_{11}^2 \sin(\theta_1 - \theta_2))\cos(\theta_1 + \theta_2) &=0,\\
    (h_{22}^1 - E_1(\theta_2)) \sin(2 \theta_2)&=0,\\
    (h_{22}^1 - E_1(\theta_2))\cos(2 \theta_2) &=0,\\
    ( h_{12}^3 \cos(\theta_2 - \theta_3) + (-\tfrac{1}{\sqrt{6} }+ \omega_{12}^3) \sin(\theta_2 - \theta_3)) \sin(\theta_2 + \theta_3)&=0,\\
    ( h_{12}^3 \cos(\theta_2 - \theta_3) + (-\tfrac{1}{\sqrt{6}} + \omega_{12}^3) \sin(\theta_2 - \theta_3))\cos(\theta_2 + \theta_3) &=0.
\end{align*}
Since sine and cosine never vanish at the same time (\ref{anglederi}) and (\ref{sffc}) hold for $i=1$ and $j=2$. The other equations follow in a similar way, by choosing different $X$ and $Y$ in (\ref{nablap}).
\end{proof}

\begin{proposition}
Let $M$ be a Lagrangian submanifold of $\Sl\times\Sl$ of diagonalizable type. If the angle functions $\theta_i$ are constant and $h_{12}^3=0$, then $M$ is totally geodesic. Conversely, if the submanifold is totally geodesic then the angles are constant.    
\end{proposition}
\begin{proof}
    By Lemma \ref{lemmacase1}, if the angle functions are constant then $h_{ii}^j=0$ for all $i,j=1,2,3$. Using the symmetries of $h_{ij}^k$ we conclude that the submanifold is totally geodesic. The converse is immediate by Lemma \ref{lemmacase1}.
\end{proof}

Now notice that
\begin{equation}
    \begin{split}
        -\delta_kE_k(h_{jj}^i)+\delta_{i}E_i(h_{jj}^k)&=\delta_k\delta_i\delta_jE_k(E_i(\theta_j))-\delta_{i}\delta_k\delta_jE_i(E_k(\theta_j)))\\
        &=\delta_k\delta_i\delta_j[E_k,E_i](\theta_j)\\
        &=\delta_k\delta_i\delta_j (\nabla_{E_k}E_i-\nabla_{E_i}E_k)(\theta_j)\\
        &=\delta_k\delta_i\delta_j\sum_l(\omega_{ki}^l-\omega_{ik}^l)E_l(\theta_j)\\
        &=\delta_k\delta_i\sum_l\delta_l(\omega_{ik}^l-\omega_{ki}^l)h_{jj}^l.\\
    \end{split}
    \label{compati}
\end{equation}

\begin{proposition}\label{twoanglesequal}
Let $M$ be a Lagrangian submanifold of $\Sl\times\Sl$ of diagonalizable type. If two angles are equal modulo $\pi$, then $M$ is totally geodesic.
\end{proposition}
\begin{proof}
We assume that $\theta_1=\theta_2$ mod $\pi$. From (\ref{anglederi}) we get $-h_{22}^i=h_{11}^i$ and $h_{i1}^2=0$, for all $i$. Thus $h_{11}^1=h_{11}^2=h_{22}^1=h_{22}^2=h_{12}^3=0$. By Proposition \ref{minimal} the submanifold $M$ is minimal, then $-h_{11}^i+h_{22}^i+h_{33}^i=0$ for $i=1,2,3$. Hence, $h_{33}^1$ and $h_{33}^2$ also vanish. The remaining components are related by $h_{11}^3=-h_{22}^3$ and  $h_{33}^3=2h_{11}^3$. Taking $i=2,j=3,k=1$ in (\ref{compati}) we obtain $0=(\omega_{21}^3-\omega_{12}^3)h_{11}^3$. Computing both sides of the Codazzi equation in (\ref{Codazzi}) with $X=E_1,Y=E_2,Z=E_2$ yields
\[
0=h_{11}^3(\sqrt{\tfrac{2}{3}}-3\omega_{12}^3+\omega_{21}^3).
\]
Suppose $h_{11}^3\neq 0$. Then $\omega_{12}^3=\omega_{21}^3=\tfrac{1}{\sqrt{6}}$ and hence by (\ref{sffc}) we have $0=(-\tfrac{1}{\sqrt{6}}-\tfrac{1}{\sqrt{6}})\sin(\theta_1-\theta_3)$, so $\theta_1=\theta_3$ mod $\pi$. 
By (\ref{anglederi}) $h_{11}^3=-h_{33}^3=0$, which is a contradiction. 
Thus, $h_{11}^3$ must be zero and therefore the submanifold is totally geodesic. 
The proof is similar (up to signs) for a different choice of the pair of angles.
\end{proof}

\begin{lemma}\label{isometrias}
 Let $f:M\to\Sl\times\Sl$ be a Lagrangian immersion into the pseudo-nearly Kähler $\Sl\times\Sl$, and assume that $M$ is of diagonalizable type with $\Delta_1$-orthonormal frame $\{E_1,E_2,E_3\}$. 
 Let $\theta_1,\theta_2,\theta_3$ be their respective angle functions. 
 Then the Lagrangian immersions $\Psi_{1,0}\circ f$ , $\Psi_{1,4\pi/3}\circ f$, where $\Psi_{\kappa,\tau}$ are the isometries given in \emph{(\ref{isoslsl})}, are also of diagonalizable type and their respective angle functions are
\[
\theta_i^{(1)}=\pi-\theta_i, \ \ \ \ \theta_i^{(2)}=\frac{\pi}{3}-\theta_i.
\]

\end{lemma}
\begin{proof}
      The proofs of Theorem 3 and Theorem 4 of \cite{constantangles} can be replicated for this case, taking the adjugate matrix instead of the conjugate of a quaternion.
\end{proof}
\begin{lemma}\label{trescasos}
Consider a totally geodesic Lagrangian submanifold of diagonalizable type of the pseudo-nearly Kähler $\Sl\times\Sl$. 
After a possible permutation of the angles, we have one of the following:
\begin{enumerate}
        \item $(2\theta_1,2\theta_2,2\theta_3)=(\tfrac{4\pi}{3},\tfrac{4\pi}{3},\tfrac{4\pi}{3})$, \label{diag1}
        \item $(2\theta_1,2\theta_2,2\theta_3)=(0,\pi,\pi)$, \label{diag2}
        \item $(2\theta_1,2\theta_2,2\theta_3)=(\pi,\pi,0)$. \label{diag3}
\end{enumerate}

\end{lemma}
\begin{proof}
      Taking $X=E_i$, $Y=E_j$ and $Z=E_j$ in the Codazzi equation yields $\sin(2(\theta_i-\theta_j))=0$, thus all pairs of angles $2\theta_i$ and $2\theta_j$ are equal up to a multiple of $\pi$. This, together with the fact that the sum of the angles is equal to zero modulo $2\pi$,  implies that $6\theta_i=0$ modulo $\pi$ for all $i$. Therefore, replacing $E_2$ with $E_3$ and $E_3$ with $-E_2$ if necessary, we obtain the following possibilities.

      \vspace{0.2 cm}
      
        \begin{tasks}[label=(\arabic*)](2)
            \task \ $(2\theta_1,2\theta_2,2\theta_3)=(\tfrac{4\pi}{3},\tfrac{4\pi}{3},\tfrac{4\pi}{3})$,\label{c1}
            \task \ $(2\theta_1,2\theta_2,2\theta_3)=(\tfrac{2\pi}{3},\tfrac{2\pi}{3},\tfrac{2\pi}{3})$,\label{c2}
            \task \ $(2\theta_1,2\theta_2,2\theta_3)=(0,0,0)$,\label{c3}
            \task \ $(2\theta_1,2\theta_2,2\theta_3)=(0,\pi,\pi)$,\label{c4}
            \task \ $(2\theta_1,2\theta_2,2\theta_3)=(\pi,\pi,0)$,\label{c5}
            \task \ $(2\theta_1,2\theta_2,2\theta_3)=(\tfrac{\pi}{3},\tfrac{\pi}{3},\tfrac{4\pi}{3})$,\label{c6}
            \task \ $(2\theta_1,2\theta_2,2\theta_3)=(\tfrac{4\pi}{3},\tfrac{\pi}{3},\tfrac{\pi}{3})$,\label{c7}
            \task \ $(2\theta_1,2\theta_2,2\theta_3)=(\tfrac{2\pi}{3},\tfrac{5\pi}{3},\tfrac{5\pi}{3})$,\label{c8}
            \task \ $(2\theta_1,2\theta_2,2\theta_3)=(\tfrac{5\pi}{3},\tfrac{5\pi}{3},\tfrac{2\pi}{3})$.\label{c9}
            \task []
        \end{tasks}

        Suppose that $M$ is a Lagrangian submanifold of diagonalizable type with angles as in \ref{c3}.
         Then by Lemma \ref{isometrias}, applying the isometry $\Psi_{1,4\pi/3}$ in (\ref{isoslsl}) produces a congruent Lagrangian submanifold with angles as in \ref{c1}. 
         We can apply a similar argument to see that cases \ref{c2}, \ref{c3}, \ref{c6}, \ref{c7}, \ref{c8} and \ref{c9} are congruent to one of the cases \ref{c1}, \ref{c4} or \ref{c5}.
          Hence we may restrict our attention to the latter cases. 
         Notice that in cases \ref{c4} and \ref{c5} we cannot do a permutation with $E_1$ since, unlike $E_2$ and $E_3$, this vector is timelike. 
         Later on we will prove that these cases are not isometric.
\end{proof}

We proceed by giving three examples of totally geodesic Lagrangian submanifolds of the nearly Kähler $\Sl\times\Sl$. 

Consider the frame $\{X_1,X_2,X_3\}$  given by 
\begin{equation}
    \begin{split}
        X_1(a)&=a\ii,
        \ \ \ \ \ 
        X_2(a)=a\jj,
        \ \ \ \ \ 
        X_3(a)=a\kk,
    \end{split}\label{frameX}
\end{equation}
for which 
\[
-\langle X_1,X_1 \rangle =\langle X_2,X_2\rangle=\langle X_3,X_3\rangle=1.
\]
where $\li,\ri$ is the metric given in (\ref{prodsl2}). Then it is a $\Delta_1$-orthonormal frame.

\begin{example} \label{example1} Consider the immersion of $\Sl$ into $\Sl\times\Sl$ given by
\[
f\colon \Sl\to \Sl\times\Sl\colon u\mapsto (\id_2,u). 
\]
Let $X_1$, $X_2$ and $X_3$ be the vector fields on $\Sl$ given in Equation (\ref{frameX}). We have
\[
df(X_1(u))=(0,u\ii)_{(\id_2,u)}, \ \ \ df(X_2(u))=(0,u\jj)_{(\id_2,u)}, \ \ \ df(X_3(u))=(0,u\kk)_{(\id_2,u)}.
\]
Then we compute 
\[
Pdf(X_1(u))=(\ii,0)_{(\id_2,u)}, \ \ \ Pdf(X_2(u))=(\jj,0)_{(\id_2,u)}, \ \ \ Pdf(X_3(u))=(\kk,0)_{(\id_2,u)},
\]
and
\begin{equation*}
\begin{split}
    Jdf(X_1(u))=-\frac{1}{\sqrt{3}}(2\ii,u\ii)_{(\id_2,u)}, \\
    Jdf(X_2(u))=-\frac{1}{\sqrt{3}}(2\jj,u\jj)_{(\id_2,u)}, \\ 
    Jdf(X_3(u))=-\frac{1}{\sqrt{3}}(2\kk,u\kk)_{(\id_2,u)}. \\
\end{split}    
\end{equation*}
From this, it immediately follows that $g(df(X_i(u)),Jdf(X_j(u)))=0$ for all $i,j=1,2,3$. Therefore $f$ is a Lagrangian immersion.
Moreover, we notice that
\[
Pdf(X_i(u))=-\tfrac{1}{2}df(X_i(u))-\tfrac{\sqrt{3}}{2}Jdf(X_i(u),
\]
hence the angle functions are constant and all are equal to $\tfrac{4}{3}\pi$, as in Case (\ref{diag1}) of Lemma \ref{trescasos}. Using Proposition \ref{twoanglesequal} we conclude that $f$ is totally geodesic.
\end{example}

\begin{example}\label{example2}
Consider the immersion of $\Sl$ into $\Sl\times\Sl$ given by 
\[f:\Sl\to\Sl\times\Sl\colon u\mapsto(u,\ii u\ii).\]
We may compute
\[
df(X_1(u))=(u\ii,-\ii u)_{(u,\ii u\ii)}, \ \ \ df(X_2(u))=(u\jj,-\ii u\kk)_{(u,\ii u\ii)}, \ \ \ df(X_3(u))=(u\kk,\ii u\jj)_{(u,\ii u\ii)}.
\]
By the definition of $J$ we get
\begin{equation*}
    \begin{split}
        Jdf(X_1(u))&=-\frac{1}{\sqrt{3}}(u\ii,\ii u)_{(u,\ii u\ii)},\\
        Jdf(X_2(u))&=\sqrt{3}(u\jj,\ii u\kk)_{(u,\ii u\ii)},\\
        Jdf(X_3(u))&=\sqrt{3}(u\kk,-\ii u\jj)_{(u,\ii u\ii)}.\\
    \end{split}
\end{equation*}
Hence, we can easily compute $g(df(X_i(u),Jdf(X_j(u)))=0$ for $i,j=1,2,3$, therefore $f$ is a Lagrangian immersion.
Moreover, we have
\begin{equation*}
    \begin{split}
        Pdf(X_1(u))&=(u\ii,-\ii u)_{(u,\ii u\ii)}=df(X_1(u)), \\
        Pdf(X_2(u))&=(-u\jj,\ii u\kk)_{(u,\ii u\ii)}=-df(X_2(u)), \\
        Pdf(X_3(u))&=(-u\kk,-\ii u\jj)_{(u,\ii u\ii)}=-df(X_3(u)).
    \end{split}
\end{equation*}
 Consequently, the angle functions are constant and equal to $(0,\pi,\pi)$, as in Case (\ref{diag2}) of Lemma \ref{trescasos}. By Proposition (\ref{twoanglesequal}) $f$ is totally geodesic.
\end{example}

\begin{example}\label{example3}
 Consider the immersion of $\Sl$ into $\Sl\times\Sl$ given by
\[
f\colon \Sl\to \Sl\times\Sl\colon u\mapsto (u,\kk u\kk). 
\]
We have
\[
df(X_1(u))=(u\ii,-\kk u\jj)_{(u,\kk u\kk )}, \ \ \ df(X_2(u))=(u\jj,-\kk u\ii)_{(u,\kk u\kk )}, \ \ \ df(X_3(u))=(u\kk,\kk u)_{(u,\kk u\kk)}.
\]
By definition of  $J$ we obtain
\begin{equation*}
    \begin{split}
        Jdf(X_1(u))&=\sqrt{3}(u\ii,\kk u\jj)_{(u,\kk u\kk)},\\
        Jdf(X_2(u))&=\sqrt{3}(u\jj,\kk u\ii)_{(u,\kk u\kk)},\\
        Jdf(X_3(u))&=\frac{1}{\sqrt{3}}(-u\kk, u\kk)_{(u,\kk u\kk)}.\\
    \end{split}
\end{equation*}
Hence we can easily compute $g(df(X_i(u),Jdf(X_j(u)))=0$ for $i,j=1,2,3$, therefore $f$ is a Lagrangian immersion. Moreover we have that
\begin{equation*}
    \begin{split}
        Pdf(X_1(u))&=(-u\ii,\kk u\jj)_{(u,\kk u\kk)}=-df(X_1(u)), \\
        Pdf(X_2(u))&=(-u\jj,\kk u\ii)_{(u,\kk u\kk)}=-df(X_2(u)), \\
        Pdf(X_3(u))&=(u\kk,\kk u)_{(u,\kk u\kk)}=df(X_3(u)).
    \end{split}
\end{equation*}
Therefore, the angle functions are constant and equal to $(\pi,\pi,0)$, corresponding to Case (\ref{diag3}) of Lemma~\ref{trescasos}. By Proposition (\ref{twoanglesequal}) $f$ is totally geodesic.
\end{example}

\begin{remark}
    The immersion $f:\Sl\to\Sl\times\Sl$ given by $f:u\mapsto (u,\jj u\jj)$ is a Lagrangian immersion congruent to Example \ref{example3}. Indeed, take the isometry of $\Sl\times\Sl$ given by $\phi:(p,q)\mapsto(cpc,cqc)$, where $c$ is the matrix in $\Sl$ given by $e^{-\tfrac{\pi}{4} \kk}$. After an easy computation we get that
    \[
    \phi\circ f (u)=(cuc,c\jj u \jj c)=(v,\kk v\kk),
    \]
    where $v=cuc$.

\end{remark}

We have now proven that the immersions in Theorem \ref{maintheorem} are Lagrangian. To complete the proof of this theorem, it remains to be shown that any totally geodesic Lagrangian immersion is locally isometric to one of the Examples \ref{example1}-\ref{example3}, depending on the possible angle functions in Lemma \ref{trescasos}. This is what we prove in the upcoming propositions.
      \begin{proposition}\label{idu}
      Let $M$ be a totally geodesic Lagrangian submanifold of the pseudo-nearly Kähler $\Sl\times\Sl$ of diagonalizable type. Assume that  $(2\theta_1,2\theta_2,2\theta_3)=(\tfrac{4\pi}{3},\tfrac{4\pi}{3},\tfrac{4\pi}{3})$. Then $M$ is locally congruent to the submanifold
      $\Sl\to\Sl\times\Sl\colon u\mapsto (\id_2,u)$. 
      \end{proposition}
      \begin{proof}
            By Case \ref{case:10.1} of Proposition \ref{propAB}, we have that $A=-\tfrac{1}{2}\id$ and $B=-\frac{\sqrt{3}}{2}$ are multiples of the identity we have that $PX=-\tfrac{1}{2}X-\tfrac{\sqrt{3}}{2}JX$ for any vector field $X$ tangent to $M$. It follows immediately that $QX=X$ where $Q$ is the almost product structure given in (\ref{prodQ}). Using the Gauss equation we can compute the curvature tensor of $M$ and also the sectional curvature, which is equal to $-\tfrac{3}{2}$. Then $M$ is locally isometric to $\Sl$ equipped with the metric $\tfrac{2}{3}g_0$,
            where $g_0$ is the metric defined in (\ref{prodsl2}). Now write $f=(p,q)$. By the definition of $Q$, we have that
            \begin{equation*}
                (dp(v),0)=\tfrac{1}{2}(df(v)-Qdf(v))=0, \ \ \ (0,dq(v))=\tfrac{1}{2}(df(v)+Qdf(v))=df(v), \ \ \ v\in T\Sl.
            \end{equation*}
            hence $p$ should be a constant matrix in $\Sl$. The previous equation also implies that $dq$ is a linear isomorphism, then $q$ is a local diffeomorphism. Therefore, we may assume that $q(u)$ is actually equal to $u$. Applying an isometry of $\Sl\times\Sl$ we may assume that $p$ is equal to $\id_2$.
      \end{proof}

The immersion in Example \ref{example2} is the immersion of $\Sl$ with a Berger-like metric, stretched in the timelike direction. 
We can construct such a metric by taking on $\Sl$ the frame $\{X_1,X_2,X_3\}$ given in (\ref{frameX}), and the metric $\tilde{g}$ given by
\[
\tilde{g}(X,Y)=\tfrac{4}{\kappa}\left(\li X,Y \ri+(1-\tfrac{4\tau^2}{\kappa})\li X,X_1\ri\li Y,X_1\ri\right),
\]
where $\kappa>0$ and $\tau$ are constants and $\li,\ri$ is the metric on $\Sl$ given in (\ref{prodsl2}). It follows from a straightforward computation that
\[
[X_1,X_2]=2X_3,\ \ \  [X_1,X_3]=-2X_2, \ \ \ [X_2,X_3]=-2X_1.
\]
We take the vector fields 
\begin{equation}
    \tilde{E}_1=\tfrac{\kappa}{4\tau}X_1,\ \tilde{E}_2=\tfrac{\sqrt{\kappa}}{2}X_2\ \text{and}\ \tilde{E}_3=\tfrac{\sqrt{\kappa}}{2}X_3,\label{lorentzbergframe}
\end{equation}
which form a $\Delta_1$-orthonormal frame on $\Sl$ with respect to the metric $\tilde{g}$. We denote the Levi-Civita connection associated to $\tilde{g}$ by $\nabla^\sim$. It follows from the Koszul formula that $\nabla^\sim_{\tilde{E}_i}\tilde{E}_i=0$ and that
\begin{align*}
        \nabla^\sim_{\tilde{E}_1}\tilde{E}_2&=(\frac{\kappa }{2 \tau }-\tau )\tilde{E}_3, &&    \nabla^\sim_{\tilde{E}_2}\tilde{E}_3=-\tau\tilde{E}_1,\\
            \nabla^\sim_{\tilde{E}_1}\tilde{E}_3&=(\tau-\frac{\kappa }{2 \tau })\tilde{E}_2, &&        \nabla^\sim_{\tilde{E}_3}\tilde{E}_1=\tau\tilde{E}_2,\\
                    \nabla^\sim_{\tilde{E}_2}\tilde{E}_1&=-\tau\tilde{E}_3, &&        \nabla^\sim_{\tilde{E}_3}\tilde{E}_2=\tau\tilde{E}_1.
\end{align*}
The following proposition is a particular case of Theorem 1.7.18 of \cite{wolf}, where it is proved for an arbitrary manifold equipped with a affine connection.
\begin{proposition}\label{propositionconstants}
    Let $N^n$ and $\tilde{N}^n$ be pseudo-Riemannian manifolds with Levi-Civita connections $\nabla$ and $\tilde{\nabla}$, respectively. Suppose that there exist constants $c_{ij}^k$, $i,j,k\in {1,\dots,n}$ such that for all $p\in N$ and $\tilde{p}\in\tilde{N}$ there exist pseudo-orthonormal frames $\{F_1,\ldots,F_n\}$ around $p$, $\{\tilde{F}_1,\dots,\tilde{F}_n\}$  around $\tilde{p}$ with the same signatures such that $\nabla_{F_i}F_j=\sum_{i=1}^n c_{ij}^kF_k$, $\tilde{\nabla}_{\tilde{F}_i}\tilde{F}_j=\sum_{i=1}^nc_{ij}^k \tilde{F}_k$. Then there exists a local isometry that maps a neighbourhood of $p$ into a neighbourhood of $\tilde{p}$ and $\{F_i\}$ to $\{\tilde{F}_i\}$.
\end{proposition}
We will  use this proposition to show that cases (\ref{diag2}) and (\ref{diag3}) of Lemma \ref{trescasos} are locally $\Sl$ with Berger-like metrics.
\begin{proposition}\label{uuk}
            Let $M$ be a totally geodesic Lagrangian submanifold of the pseudo-nearly Kähler $\Sl\times\Sl$, of diagonalizable type. Assume that $(2\theta_1,2\theta_2,2\theta_3)=(0,\pi,\pi)$. Then $M$ is locally isometric to the submanifold $\Sl\to\Sl\times\Sl\colon u\mapsto (u,\ii u\ii)$. 
\end{proposition}
\begin{proof}
        Let $\{E_1,E_2,E_3\}$ be a $\Delta_1$-orthonormal frame such that $JG(E_1,E_2)=\sqrt{\tfrac{2}{3}}E_3$ and $A$ and $B$ take the form in \ref{case:10.1} of Lemma \ref{propAB} with $(2\theta_1,2\theta_2,2\theta_3)=(0,\pi,\pi)$. Thus
        \begin{equation}
             PE_1=E_1, \ \ \ \ PE_2=-E_2, \ \ \ PE_3=-E_3.      \label{Pzeropipi}
        \end{equation}
        From Lemma \ref{lemmacase1} we obtain that $\omega_{i1}^j=0$ for $i,j=1,2,3$. We also get
\[
\omega_{21}^3=-\frac{1}{\sqrt{6}},\ \ \ \  \omega_{32}^1=\frac{1}{\sqrt{6}}.
\]

From the Gauss Equation (\ref{Gauss}) we obtain the following equations:
\begin{equation}
    \begin{split}
       E_1(\omega_{22}^3)-E_2(\omega_{12}^3)+\left(\omega_{12}^3+\frac{1}{\sqrt{6}}\right) \omega_{33}^2=0,\\
       -E_1(\omega_{33}^2)-E_3(\omega_{12}^3)+\left(\omega_{12}^3+\frac{1}{\sqrt{6}}\right) \omega_{22}^3=0,\\
       E_2(\omega_{33}^2)+E_3(\omega_{22}^3)-\sqrt{\frac{2}{3}} \omega_{12}^3-(\omega_{22}^3)^2-(\omega_{33}^2)^2+\frac{5}{3}=0.\\
    \end{split}\label{equationstg0pipi}
\end{equation}
Define the 1-form $\omega$ by 
\[
\omega(E_1)=-\omega_{12}^3+\frac{5}{\sqrt{6}}, \ \ \ \omega(E_2)=-\omega_{22}^3, \ \ \ \ \omega(E_3)=\omega_{33}^2.
\]
Using (\ref{equationstg0pipi}) we can prove that $\omega$ is closed. Hence there exists a local function $\varphi$ such that $d\varphi=\omega$.
Now define the new frame 
\[
F_1=-E_1,\ \ \ \  F_2=-\cos\varphi E_2-\sin\varphi E_3,\ \ \ \ \ F_3=\sin\varphi E_2-\cos\varphi E_3.
\]
 This new frame is still $\Delta_1$-orthonormal which satisfies (\ref{Pzeropipi}) and $JG(F_1,F_2)=\sqrt{\tfrac{2}{3}}F_3$.  We have
\begin{align*}
            \nabla_{F_1}F_1&=0&&         \nabla_{F_1}F_2=-\frac{5}{\sqrt{6}}F_3,&&         \nabla_{F_1}F_3=\frac{5}{\sqrt{6}}F_2,\\
            \nabla_{F_2}F_1&=\frac{1}{\sqrt{6}}F_3, &&       \nabla_{F_2}F_2=0,  &&        \nabla_{F_2}F_3=\frac{1}{\sqrt{6}}F_1,\\
            \nabla_{F_3}F_1&=-\frac{1}{\sqrt{6}}F_2,  && \nabla_{F_3}F_2=-\frac{1}{\sqrt{6}}F_1, &&  \nabla_{F_3}F_3=0.
\end{align*}
By Proposition \ref{propositionconstants} we have that $M$ is locally isometric to $\Sl$ with a Berger-like metric stretched in the timelike direction  with $\tau=-\frac{1}{\sqrt{6}}$ and $\kappa=2$.
Now using (\ref{relprodkal}) we may write
\begin{align}
      \nabla^E _{F_1}F_1&=0, &&  \nabla^E _{F_1}F_2=-\sqrt{\frac{3}{2}}F_3, &&        \nabla^E _{F_1}F_3=\sqrt{\frac{3}{2}}F_2,\nonumber \\
      \nabla^E _{F_2}F_1&=\sqrt{\frac{3}{2}}F_3, &&    \nabla^E _{F_2}F_2=0, &&         \nabla^E _{F_2}F_3=\frac{1}{\sqrt{6}}F_1, \label{connF}\\
      \nabla^E _{F_3}F_1&=-\sqrt{\frac{3}{2}}F_2, && \nabla^E _{F_3}F_2=-\frac{1}{\sqrt{6}}F_1, &&         \nabla^E _{F_3}F_3=0.\nonumber
\end{align}
where $\nabla^E$ is the Levi-Civita connection associated to the product metric. We can identify the frame $\{F_i\}_i$ on $\Sl$ with the frame given in (\ref{lorentzbergframe}), i.e.,
\begin{equation}
     F_1=-\sqrt{\frac{3}{2}}X_1,\ \ \ \  F_2=\frac{1}{\sqrt{2}}X_2, \ \ \ \ F_3=\frac{1}{\sqrt{2}}X_3.\label{identificationberger}
\end{equation}

Now writing the immersion $f=(p,q)$ and $df(F_i)=D_{F_i}f=(p\alpha_i,q\beta_i)$, where $\alpha_i,\beta_i$ are matrices in $\mathfrak{sl}(2,\R)$, we obtain
\begin{equation}
\beta_1=\alpha_1, \ \ \ \ \ \beta_2=-\alpha_2, \ \ \ \ \ \beta_3=-\alpha_3.\label{eqalfabeta}
\end{equation}
because of Equation (\ref{Pzeropipi}).

It follows from (\ref{prodmetric}) that $\alpha_i$ are mutually orthogonal and their lengths are given by
\[
\li \alpha_1,\alpha_1\ri=-\frac{3}{2}, \ \ \ \ \li\alpha_2,\alpha_2\ri=\li\alpha_3,\alpha_3\ri=\frac{1}{2}.
\]
Thus we can write the Lorentzian cross products as
\begin{equation*}
    \alpha_1\times\alpha_2=\varepsilon \sqrt{\frac{3}{2}}\alpha_3, \ \ \ \alpha_2\times\alpha_3=-\varepsilon \frac{1}{\sqrt{6}}\alpha_1,  \ \ \ \ \ \alpha_3\times\alpha_1=\varepsilon \sqrt{\frac{3}{2}}\alpha_2,
\end{equation*}
where $\varepsilon=\pm1$.
We compute
\[
D_{F_i}D_{F_j}f=(p\alpha_i\times\alpha_j+\li\alpha_i,\alpha_j\ri p+pF_i(\alpha_j),\ q\beta_i\times\beta_j+\li\beta_i,\beta_j\ri q+q F_i(\beta_j)),
\]
therefore applying (\ref{connectionr8}) it follows that
\[
\nabla^E_{F_i}F_j=(p\alpha_i\times\alpha_j+pF_i(\alpha_j),\ q\beta_i\times\beta_j+ q F_i(\beta_j)).
\]
Comparing the above equation with (\ref{connF}) we obtain
\begin{align*}
        F_1(\alpha_1)&=0,     &&    F_2(\alpha_1)=\sqrt{\frac{3}{2}}(1+\varepsilon)\alpha_3, && F_3(\alpha_1)=-\sqrt{\frac{3}{2}}(1+\varepsilon)\alpha_3,\\
    F_1(\alpha_2)&=-\sqrt{\frac{3}{2}}(1+\varepsilon)\alpha_3, &&    F_2(\alpha_2)=0,&&  F_3(\alpha_2)=-\frac{1}{\sqrt{6}}(1+\varepsilon)\alpha_1,\\ 
    F_1(\alpha_3)&=\sqrt{\frac{3}{2}}(1+\varepsilon)\alpha_2,&&    F_2(\alpha_3)=\frac{1}{\sqrt{6}}(1+\varepsilon)\alpha_1,&&    F_3(\alpha_3)=0. 
\end{align*}
Making use of (\ref{identificationberger}) yields
\begin{align*}
      X_1(\alpha_1)&=0, &&    X_2(\alpha_1)=\sqrt{3}(1+\varepsilon)\alpha_3, &&     X_3(\alpha_1)=-\sqrt{3}(1+\varepsilon)\alpha_3,\\
          X_1(\alpha_2)&=(1+\varepsilon)\alpha_3, &&    X_2(\alpha_2)=0,    &&  X_3(\alpha_2)=-\frac{1}{\sqrt{3}}(1+\varepsilon)\alpha_1,\\
              X_1(\alpha_3)&=-(1+\varepsilon)\alpha_2, &&    X_2(\alpha_3)=\frac{1}{\sqrt{3}}(1+\varepsilon)\alpha_1, &&    X_3(\alpha_3)=0.
\end{align*}
We can write the same equations for $\beta_i$. Taking into account (\ref{eqalfabeta}), we conclude $\varepsilon$ must be equal to $-1$. Therefore $\alpha_i$ is constant for all $i$. 
We can choose an isometry on $\Sl$, namely conjugation by a matrix $c$ in $\Sl$, such that
\begin{equation*}
    \alpha_1=-\sqrt{\frac{3}{2}}c\ii c^{-1},\ \ \ \  \alpha_2=\frac{1}{\sqrt{2}}c\jj c^{-1},\ \ \ \  \alpha_3=\frac{1}{\sqrt{2}}c\kk c^{-1}.
\end{equation*}
After choosing initial conditions $f(\id_2)=(\id_2,\id_2)$, we obtain that the unique solution of the system $D_{F_i}f=(p\alpha_i,q\beta_i)$ is $p(u)=cuc^{-1}$, $q(u)=-c\ii u\ii c^{-1}$. Taking an isometry on the pseudo-nearly Kähler $\Sl\times\Sl$ finally yields $p(u)=u$ and $q(u)=\ii u\ii$.
\end{proof}

The third example of a totally geodesic immersion is locally isometric to $\Sl$ with a Berger-like metric, which is stretched in the direction of a spacelike component. We can construct such a metric by taking on $\Sl$ the frame $\{X_1,X_2,X_3\}$, and the metric $\tilde{g}$ given by
\[
\tilde{g}(X,Y)=\tfrac{4}{\kappa}\left(\li X,Y \ri+(\tfrac{4\tau^2}{\kappa}-1)\li X,X_3\ri\li Y,X_3\ri\right),
\]
where $\kappa$, $\tau$ are constants and $\li,\ri$ is the metric on $\Sl$ given in (\ref{prodsl2}). It follows from a straightforward computation that
\[
[X_1,X_2]=2X_3,\ \ \  [X_1,X_3]=-2X_2, \ \ \ [X_2,X_3]=-2X_1.
\]
We take the vector fields $\tilde{E}_1=\tfrac{\sqrt{\kappa}}{2}X_1$, $\tilde{E}_2=\tfrac{\sqrt{\kappa}}{2}X_2$ and $\tilde{E}_3=\tfrac{\kappa}{4\tau}X_3$, which form a pseudo-orthonormal frame on $\Sl$ with respect to the metric $\tilde{g}$. We denote $\nabla^\sim$ as the Levi-Civita connection associated to $\tilde{g}$. It follows from the Koszul formula that $\nabla^\sim_{\tilde{E}_i}\tilde{E}_i=0$ and that
\begin{align*}
    \nabla^\sim_{\tilde{E}_1}\tilde{E}_2&=\tau\tilde{E}_3, && \nabla^\sim_{\tilde{E}_2}\tilde{E}_3=-\tau\tilde{E}_1,\\
    \nabla^\sim_{\tilde{E}_1}\tilde{E}_3&=-\tau\tilde{E}_2, && \nabla^\sim_{\tilde{E}_3}\tilde{E}_1=(\frac{\kappa}{2\tau}-\tau)\tilde{E}_2,\\
        \nabla^\sim_{\tilde{E}_2}\tilde{E}_1&=-\tau\tilde{E}_3, && \nabla^\sim_{\tilde{E}_3}\tilde{E}_2=(\frac{\kappa}{2\tau}-\tau)\tilde{E}_1.
\end{align*}

\begin{proposition}\label{ujuj}
            Let $M$ be a totally geodesic Lagrangian submanifold of the pseudo-nearly Kähler $\Sl\times\Sl$. Assume that $(2\theta_1,2\theta_2,2\theta_3)=(\pi,\pi,0)$. Then $M$ is locally isometric to the submanifold $\Sl\to\Sl\times\Sl\colon u\mapsto (u,\kk u\kk)$. 
\end{proposition}

\begin{proof}The proof is similar to the proof of Proposition \ref{uuk}, but with some minor differences. First we take a frame $\{E_1,E_2,E_3\}$ such that $A$ and $B$ take the form of Case \ref{case:10.1} of Lemma \ref{propAB}. Then we take the closed 1-form given by
\[
\omega(E_1)=-\omega_{11}^2, \ \ \ \omega(E_2)=-\omega_{22}^1, \ \ \ \ \omega(E_3)=-\omega_{32}^1+\frac{5}{\sqrt{6}}.
\]
We define the frame $F_i$ as
\[
F_1=\cosh(\varphi) E_1+\sinh(\varphi)E_2,\ \ \ \  F_2=\sinh(\varphi)E_1+\cosh(\varphi)E_2,\ \ \ \ \ F_3=-E_3,
\]
where $\varphi$ is a local function such that $d\varphi=\omega$.

As before we write $df(F_i)=D_{F_i}f=(p\alpha_i,q\beta_i)$, where $\alpha_i,\beta_i$ are matrices in $\mathfrak{sl}(2,\R)$. They satisfy $\alpha_1=-\beta_1$, $\alpha_2=-\beta_2$ and $\alpha_3=\beta_3$. Again we may choose $c$ in $\Sl$, such that
\begin{equation*}
    \alpha_1=\frac{1}{\sqrt{2}}c\ii c^{-1},\ \ \ \  \alpha_2=\frac{1}{\sqrt{2}}c\jj c^{-1},\ \ \ \  \alpha_3=-\sqrt{\frac{3}{2}}c \kk c^{-1}.
\end{equation*}
After choosing initial conditions $f(\id_2)=(\id_2,\id_2)$, we obtain that the unique solution of the system $D_{F_i}f=(p\alpha_i,q\beta_i)$ is $p(u)=cuc^{-1}$, $q(u)=c\kk u\kk c^{-1}$. Taking an isometry on the pseudo-nearly Kähler $\Sl\times\Sl$ finally yields $p(u)=u$, $q(u)=\kk u\kk$.
\end{proof}


\subsection{Lagrangian submanifolds of the non-diagonalizable types}
In this subsection we show that there do not exist totally geodesic Lagrangian submanifolds in types (2), (3) and (4) of Lemma \ref{propAB}.

\begin{lemma}\label{caseee2}
There are no totally geodesic immersions into the nearly Kähler $\Sl\times\Sl$ corresponding to Case \emph{\ref{case:10.2}} of Lemma \emph{\ref{propAB}}.
\end{lemma}
\begin{proof}
Suppose that $M$ is a totally geodesic Lagrangian submanifold of $\Sl\times\Sl$ associated to Case \ref{case:10.2} of Lemma \ref{propAB}. The left hand side of the Codazzi equation in (\ref{Codazzi}) for a totally geodesic submanifold is always zero. Computing the right hand side for $X=E_1$, $Y=E_2$ and $Z=E_2$ yields  
\[
-\frac{4}{3} ( \sin 2 \theta_1+ \cos 2 \theta_1 \cot 2 \theta_1)JE_1
\]
which cannot be zero, therefore a contradiction.
\end{proof}

\begin{lemma}\label{caseee3}
There are no totally geodesic immersions into the nearly Kähler $\Sl\times\Sl$ corresponding to Case \emph{\ref{case:10.3}} of Lemma \emph{\ref{propAB}}.
\end{lemma}
\begin{proof}
Suppose that $M$ is a totally geodesic Lagrangian submanifold of $\Sl\times\Sl$ associated to Case \ref{case:10.3} of Lemma \ref{propAB}. 
As in the previous lemma, the left hand side  of the Codazzi equation is zero. 
The component in the direction of $JE_1$ of the right hand side of the Codazzi equation for Case \ref{case:10.3} with $X=E_1,Y=E_2,Z=E_2$ is $\pm\tfrac{8}{9 \sqrt{3}}$, which is  a contradiction.  
\end{proof}

\begin{lemma}\label{caseee4}
There are no totally geodesic immersions into the nearly Kähler $\Sl\times\Sl$ corresponding to Case \emph{\ref{case:10.4}} of Lemma \emph{\ref{propAB}}.
\end{lemma}
\begin{proof}
Suppose that $M$ is a totally geodesic Lagrangian submanifold of $\Sl\times\Sl$ associated to Case \ref{case:10.4} of Lemma \ref{propAB}. The left hand side of the Codazzi equation is always zero for totally geodesic submanifolds and the right hand side of the Codazzi equation for $X=E_1, Y=E_2,Z=E_2$ is
\[
\frac{2}{3} \sinh (2 \lambda ) \cos (2 \theta_1+\theta_2)JE_2
\]
which is not zero since $\lambda$ must be different from zero and $2\theta_1+\theta_2$ is equal to zero modulo $\pi$.
\end{proof}
\subsection{Proof of the main result}
We adapt a lemma from \cite{reckziegel} to Lagrangian submanifolds:
 \begin{lemma}\label{michael}
     Let $\mathcal{F}_1,\mathcal{F}_2\colon M\to \Sl\times\Sl$ 
     be two Lagrangian immersions into the pseudo-nearly Kähler $\Sl\times\Sl$ of diagonalizable type. Let $E_i$ be a $\Delta_1$-orthonormal frame on $M$. Let $F_i$ and $\tilde{F}_i$ be frames along $\mathcal{F}_1$ and $\mathcal{F}_2$, respectively, such that $d\mathcal{F}_1(E_i)=F_i$ and $d\mathcal{F}_2(E_i)=\tilde{F}_i$. If $\mathcal{F}_1(p_o)=\mathcal{F}_2(p_o)$, $F_i(p_o)=\tilde{F}_i(p_o)$ and $\omega_{ij}^k=\tilde{\omega}_{ij}^k$, $h_{ij}^k=\tilde{h}_{ij}^k$, then $\mathcal{F}_1$ and $\mathcal{F}_2$ are locally congruent.
     \end{lemma}

   \begin{proof}[Proof of Theorem \ref{maintheorem}]
     In Examples \ref{example1}, \ref{example2} and \ref{example3} we showed that the three immersions of Theorem \ref{maintheorem} are totally geodesic and Lagrangian. 
     
     Let $M$ be a totally geodesic Lagrangian submanifold of the pseudo-nearly Kähler $\Sl\times\Sl$.
     By Lemma \ref{propAB}, there are four cases to consider.
     In Lemmas~\ref{caseee2}, \ref{caseee3} and \ref{caseee4} we proved that there are no totally geodesic Lagrangian submanifolds in Cases \ref{case:10.2}, \ref{case:10.3} and \ref{case:10.4} of Lemma \ref{propAB}.
     
     In Lemma \ref{trescasos} we have seen that any totally geodesic Lagrangian submanifold of $\Sl\times\Sl$ of diagonalizable type is congruent to a submanifold with angle functions $(\tfrac{4\pi}{3},\tfrac{4\pi}{3},\tfrac{4\pi}{3})$, $(\pi,\pi,0)$ or $(0,\pi,\pi)$.
     In Proposition \ref{idu}, we showed that all totally geodesic Lagrangian submanifolds with angle functions $(\tfrac{4\pi}{3},\tfrac{4\pi}{3},\tfrac{4\pi}{3})$ are locally congruent.
     The totally geodesic Lagrangian submanifolds with angle functions $(0,\pi,\pi)$ or $(\pi,\pi,0)$ are classified up to isometries by Propositions \ref{uuk} and~\ref{ujuj}. 

     It remains to prove that any totally geodesic Lagrangian submanifold of $\Sl\times\Sl$ of diagonalizable type with angle functions $(0 ,\pi,\pi)$ is congruent to Example $\ref{example2}$ and those with angle functions $(\pi,\pi,0)$ are congruent to Example $\ref{example3}$.
     Before proving this, we will show that if $(a_o\alpha,b_o\beta)$  is a tangent vector on $M$ which is part of the $\Delta_1$-orthonormal frame that diagonalizes $P$ in Lemma~\ref{propAB}, then $\alpha$ and $\beta$ are linearly dependent.
     
          Suppose that $\alpha$ and $\beta$ are linearly independent.
     Then on the one hand we have
     \[
     P(a_o\alpha,b_o\beta)=(a_o\beta,b_o\alpha)
     \]
     and on the other hand
     \[
   P(a_o\alpha,b_o\beta)=\cos 2\theta_i (a_o\alpha,b_o\beta)+\frac{\sin 2\theta_i }{\sqrt{3}}(a_o (\alpha-2\beta),b_o(2\alpha-\beta)).
     \]
Thus, we get that $\frac{2 \sin 2 \theta_i}{\sqrt{3}}=1$ and $\frac{2 \sin 2 \theta_i}{\sqrt{3}}=-1$, a contradiction. Therefore $\alpha$ and $\beta$ are multiples.

     We use Lemma \ref{michael} to prove that $M$ is locally congruent to one of the three examples of Theorem~\ref{maintheorem}.
     We have to show that given a totally geodesic Lagrangian immersion $f\colon M\to\Sl\times\Sl$, $p_o\in M$, $\{\tilde{E}_1,\tilde{E}_2,\tilde{E}_3\}$ an $\Delta_1$-orthonormal frame at $p_o$, there exists an isometry $\phi$ of $\Sl\times\Sl$ such that $\phi(f(p_o))=(\id_2,\id_2)$, the frame $d(\phi\circ f)(\tilde{E}_i)$ is equal to one of the frames in Examples \ref{example2} and \ref{example3} and the components of the second fundamental form and of the connection coincide for one of the immersions.
     Suppose that $f(p_o)=(a_o,b_o)$.
     Take the isometry $\phi\colon (a,b)\mapsto (c a_o^{-1}
     \, a \, c^{-1},c b_o^{-1} \, b\, c^{-1})$ with $c\in\Sl$, then $\phi(a_o,b_o)=(\id_2,\id_2)$ and $d\phi(a_o\alpha,b_o\beta)=(c\alpha\, c^{-1},c\beta \, c^{-1})$ for any $\alpha,\beta\in\slf$.
     We can assume that $\beta=\varepsilon \alpha$, where $\varepsilon$ is equal to $0$, $1$ or $-1$. 
     As conjugation by an element of $\Sl$ is an isometry of $\slf$, we can choose $c$ in a convenient way, taking $\alpha$ to an arbitrary element of $\slf$. Therefore, we can take the frame at $(a_o,b_o)$ to one of the frames in Examples \ref{example2} and \ref{example3}, depending on the value of $\varepsilon$.

The last step is to prove that the components of the second fundamental form and of the connection are equal to those from Examples \ref{example2} and \ref{example3}.
As the submanifold is totally geodesic, the functions $h_{ij}^k$ are all equal to zero. The components of the connection $\omega_{ij}^k$ are determined as we can see for the frame $F_i$ in the proofs of Propositions \ref{uuk} and \ref{ujuj}.



  \end{proof}
  
\section*{Acknowledgements}
The authors would like to Professor Luc Vrancken for the valuable discussions regarding this work.

\bibliographystyle{abbrv}
\bibliography{tg2.bib}


\end{document}